\documentclass[a4paper,11pt,parskip=half]{scrartcl}

\usepackage[utf8]{inputenc}
\usepackage{amsmath,amssymb,amsthm,mathtools,bm}
\usepackage{graphicx}
\usepackage{subcaption}

\usepackage{hyperref}
\usepackage{comment}

\usepackage{rotating}
\usepackage{multirow}

\usepackage[vlined,ruled]{algorithm2e}

\usepackage{pgfplots}

\newcommand{\mario}[1]{{\color{blue}{#1}}}
\newcommand{\chris}[1]{{\color{purple}{#1}}}

\providecommand{\mix}{\mathrm{mix}}
\newcommand{\rd}{\,\mathrm{d}}

\providecommand{\N}{\ensuremath{\mathbb{N}}}
\providecommand{\Z}{\ensuremath{\mathbb{Z}}}
\providecommand{\Q}{\ensuremath{\mathbb{Q}}}
\providecommand{\R}{\ensuremath{\mathbb{R}}}
\providecommand{\C}{\ensuremath{\mathbb{C}}}

\providecommand{\cH}{\ensuremath{\mathcal{H}}}

\newcommand{\bsx}{{\boldsymbol{x}}}	
\newcommand{\bsy}{{\boldsymbol{y}}}	

\newcommand{\bsm}{{\boldsymbol{m}}}
\newcommand{\bsk}{{\boldsymbol{k}}}

\newcommand{\supp}{\operatorname{supp}}
\newcommand{\spanZ}{\operatorname{span}_{\Z}}

\newcommand{\imag}{\mathrm{i}}
\newcommand{\euler}{\mathrm{e}}

\newcommand{\SLdZ}{\mbox{SL}_d(\Z)}

\newcommand{\GLdR}{\mbox{GL}_d(\R)}

\newcommand{\Nm}{\mathrm{Nm}}

\newcommand{\vol}{\mathrm{vol}}

\newcommand{\mr}[1]{\mathring{#1}}
\newcommand{\abs}[1]{\left|{#1}\right|}
\newcommand{\norm}[1]{\left\|{#1}\right\|}

\newcommand{\bsr}{{\boldsymbol{r}}}

\newcommand{\scalingFactor}{n}

\author{Christopher Kacwin \and Jens Oettershagen \and Mario Ullrich \and Tino Ullrich}
\title{Numerical performance of optimized Frolov lattices in tensor product reproducing kernel Sobolev spaces}

\newtheorem{theorem}{Theorem}[section]
\newtheorem{proposition}[theorem]{Proposition}
\newtheorem{lemma}[theorem]{Lemma}

\theoremstyle{definition}
\newtheorem{definition}[theorem]{Definition}

\theoremstyle{remark}

\newtheorem{remark}[theorem]{Remark}

\numberwithin{equation}{section}

\begin{document}

\maketitle

\begin{abstract}
	In this paper, we deal with several aspects of the universal Frolov cubature method, that is known to achieve optimal asymptotic convergence rates 
	in a broad range of function spaces. 
	Even though every admissible
	lattice has this 
	favorable asymptotic behavior, there are significant differences concerning the precise numerical behavior of the 
	worst-case error. To this end, we propose new generating polynomials that promise a significant reduction of the integration error 
	compared to the classical polynomials. Moreover, we develop a new algorithm to enumerate the Frolov points from non-orthogonal 
	lattices for numerical cubature in the $d$-dimensional unit cube $[0,1]^d$. Finally,
	we study Sobolev spaces with anisotropic mixed smoothness and compact support in $[0,1]^d$ and derive 
 	explicit formulas for their reproducing kernels. This allows for the simulation of exact worst-case errors 
	which numerically validate our theoretical results.
	
\end{abstract}

\section{Introduction}
Many scientific approaches that are related to the treatment of real world phenomena rely on the computation of integrals on high-dimensional domains which often cannot be treated analytically. Examples include physics \cite{binderbook}, computational finance \cite{Glasserman:2003}, econometrics \cite{monfortbook} and machine learning \cite{avron2016quasi,ijcai2017-207,rahimi}. In this paper, we aim for efficient and stable numerical methods to approximately 
compute the integral
\[
I_d(f) \,:=\, \int_{[0,1]^d} f(\bsx) \rd\bsx
\]
and give reliable error guarantees for a class $F_d$ of $d$-variate functions. In fact, we are particularly interested
in the worst-case error
\begin{equation}\label{eNFd}
  e(\scalingFactor,F_d) := \sup\limits_{\|f\|_{F_d} \leq 1} |I_d(f) - Q_\scalingFactor^d(f)|, 
\end{equation}
for special cubature formulas of type
\begin{equation}\label{eq:Q}
Q^d_\scalingFactor(f) \,=\, \frac{1}{\scalingFactor}\sum_{\bsk \in\Z^d}\, f(\bm{A}_\scalingFactor \bsk),
\end{equation}
where $\bm{A}_\scalingFactor:=\scalingFactor^{-1/d} \bm{A}$ is a suitable $d\times d$-matrix with $\det(\bm{A})=1$. This type of
cubature rule has a long history going back to the 1970s, see Frolov \cite{Fr76}. In \eqref{eq:Q} the function $f$ 
is assumed to be supported on a bounded domain $\Omega$ such that only finitely many summands contribute to the sum. Frolov 
noticed that the property
\begin{equation}\label{Nm}
\Nm (\bm{A}) \coloneqq
 \inf_{\bm{k}\in\Z^d\setminus\{0\}} \Big|\prod_{i=1}^d (\bm{Ak})_i\Big| > 0
\end{equation}
guarantees an optimal asymptotic worst-case behavior of \eqref{eNFd} with respect to functions 
with $L_p$-bounded mixed derivative of order $r \in \N$ supported in $[0,1]^d$. In this context, optimality 
means that the worst-case error \eqref{eNFd} can not be improved in the order sense by any other cubature 
formula using the same number of points. Note, that in case $|\det \bm{A}| = 1$ it can be shown that 
\begin{equation} \label{eqn_scalingVolume}
 \scalingFactor^{-1}\abs{\{\bsk\colon \bm{A}_\scalingFactor \bsk\in\Omega\}}\to 1
\end{equation}
for every set $\Omega$ with (Lebesgue) volume $1$ \cite{Sk94}.

Frolov showed that the set of matrices satisfying \eqref{Nm} is not empty. Moreover, he gave a 
rather sophisticated number theoretic construction with a lot of potential for numerical analysis, as we will see in this
paper. Starting with the irreducible (over $\mathbb{Q}$) 
polynomial 
\begin{equation} \label{eqn_classical_polynomial}
	P_d(x) = \prod_{j=1}^d (x-2j+1)-1 = \prod_{i=1}^d (x-\xi_i)
\end{equation}
he defined the Vandermonde matrix 
\begin{equation}\label{f21}
    \bm{V} = \left(\begin{array}{cccc}
                1 & \xi_1 & \cdots & \xi_1^{d-1}\\
                1 & \xi_2 & \cdots & \xi_2^{d-1}\\
                \vdots & \vdots & \ddots & \vdots\\
                1 & \xi_{d} & \cdots & \xi_d^{d-1}
              \end{array}\right)\,\in\GLdR\,.
\end{equation}
One reason for the increasing interest in Frolov's cubature rule is certainly the fact that once a good matrix
\eqref{Nm} is fixed the integration nodes are simply given as the rescaled image of the integer lattice points $\Z^d$
under the matrix $\bm{V}$. The method is therefore comparably \emph{simple}. Another striking aspect is a property which is 
sometimes called \emph{universality}. The method \eqref{eq:Q} is not designed for a specific class of functions 
$F_d$ as it is often the case for the commonly used quasi-Monte Carlo methods based on digital nets. In other words, we do not need 
to incorporate any a priori knowledge about the integrand (e.g. mixed or isotropic regularity etc.).

In this paper we are interested in an efficient implementation and the numerical performance of different 
Frolov type cubature methods for functions on $[0,1]^d$. First of all, this requires the efficient 
enumeration of Frolov lattice nodes in axis parallel boxes. It turned out that this is a
highly non-trivial task which has been already considered by several authors 
\cite{Kac16}, \cite{KaOeUl17}, \cite{SuYo16} including three of the 
present ones.
With a naive approach one may need to touch much more integer lattice points $\bsk\in \Z^d$ (overhead)
to check whether $\bm{A}\bsk \in [0,1]^d$. This increases the runtime of an enumeration algorithm drastically in high 
dimensions. Here, the chosen irreducible polynomial for \eqref{f21} has a significant effect. In \cite{KaOeUl17} the 
authors observed that for $d=2^m$ Chebyshev polynomials lead to an orthogonal lattice and an equivalent (orthogonal) 
lattice representation matrix with entries smaller than two in modulus. By exploiting rotational 
symmetry properties the mentioned overhead can be reduced and the enumeration procedure is less costly. 

This observation already indicated that the choice of the polynomials in \eqref{f21} is crucial. Unfortunately, 
Chebyshev polynomials and corresponding Vandermonde matrices \eqref{f21} only provide \eqref{Nm} if $d = 2^m$. This 
has been shown for instance in Temlyakov \cite{tem93}. The question remains how to fill the gaps. The classical Frolov 
polynomials are inappropriate in two respects. First, its roots spread in the range $[-d,d]$ such that \eqref{f21} gets 
highly ill-conditioned. And secondly, although the lattice satisfies \eqref{Nm}, the points are not really ``spaces filling'' 
meaning that the points accumulate around a lower dimensional manifold. This has a severe numerical impact for the 
worst-case error. In fact, the asymptotic rate of convergence is optimal but the preasymptotic behavior is useless for 
any practical issues. 

One of the main contributions of the paper is the list of new improved generating polynomials given in Section 3 below.
We give polynomials which are optimized according to the mentioned issues in dimensions $d=1,...,10$, especially with a 
narrow distribution of its roots. As already mentioned above Chebyshev polynomials itself are not irreducible if $d$ is not a power of 
two. However, they may provide admissible factors. This is the main idea of the construction and works if $d \in 
\{2,3,4,5,6,8,9,10\}$. As for the case $d=7$ a brute force search led to a polynomial with roots in $(-2.25,1.75)$.

Due to the mentioned \emph{universality} of Frolov's cubature rule, it is enough to fix the matrix and the corresponding 
lattice once and for all. In fact, the point construction does note depend on the respective framework. Therefore, it makes sense 
to generate the lattice points in a preprocessing step and make them available for practitioners. 
Our enumeration algorithm is similar to the one in \cite{KaOeUl17} and extends to non-orthogonal lattices 
by exploiting a $QR$-factorization, see Section 4. Based on the above list of 
polynomials we generated a database of Frolov lattice nodes for dimensions up to $d=10$ and $N \approx 10^6$ 
points.  The points 
are available for download and direct use on the website
\begin{center}
 \texttt{http://wissrech.ins.uni-bonn.de/research/software/frolov/}
\end{center}

Having generated the cubature points we are now able to test the 
performance of various Frolov methods for functions with bounded mixed (weak) derivative, i.e,
\begin{equation}\label{mixed}
  \langle f^{(\bsr)}, f^{(\bsr)}\rangle_{L_2} = \|f^{(\bsr)}\|_2^2 < \infty\,.
\end{equation}
where $\bsr = (r_1,...,r_d) \in \N^d$ is a smoothness vector with integer components satisfying
\begin{equation}\label{sm_vector}
    r=r_1 = ... =r_\nu < r_{\nu+1} \leq r_{\nu+2} \leq ... \leq r_d\,.
\end{equation}
A natural assumption, see 
\eqref{eq:Q}, is the restriction to functions $f$ supported inside the unit cube $\Omega = [0,1]^d$ satisfying 
\eqref{mixed}. In this case the semi-norm \eqref{mixed} becomes a norm and the corresponding space a Hilbert space
which will be denoted with $\mathring{H}^{\mathbf{r}}_{\text{mix}}$\,.

The nowadays 
well-known worst-case error 
\begin{equation}\label{f22}
  e(\scalingFactor,\mathring{H}^{\mathbf{r}}_{\text{mix}}) \asymp n^{-r}(\log n)^{(\nu-1)r},
\end{equation}
has been established in many classical papers \cite{Fr76}, \cite{Du92,Du97}, \cite{tem93}, see 
also the more recent papers \cite{MU16} and \cite{UlUl16}\,. Note, that we encounter 
another aspect of the \emph{universality} 
property for this particular framework of \emph{anisotropic mixed smoothness}. When using for instance a sparse
grid approach (see e.g. Appendix A) for the numerical integration one has to know which direction is ``rough'' in the above sense to 
adapt the sparse grid accordingly. 
In fact, one samples more points in rough directions and less points in smoother direction. Frolov's method does
not need this a priori information and behaves according to the optimal rate of convergence given in \eqref{f22}.

We will again provide a streamlined and self-contained proof in Section 6 pointing explicitly on the dependence of the constants on the dimension $d$, since the rate of convergence given by \eqref{f22} completely hides this 
dependence. In fact, in case of one minimal smoothness component in \eqref{sm_vector} even the logarithm disappears completely and we 
have a pure polynomial rate as in the univariate setting. In Theorem \ref{thm:theor-error} below 
we give a worst-case error bound which shows the influence of the dimension $d$. In addition, the result illustrates 
how the lattice invariants, like the polynomial discriminant $D_P$ and the $\ell_\infty$-diameter of the smallest  
fundamental cell enter the error estimates.

Since $\mathring{H}^{\mathbf{r}}_{\text{mix}}$ is embedded into the space of continuous functions a reproducing 
kernel exists \cite{aronszajn}. We use the approach of Wahba \cite{Wah90} as a starting point 
to derive its reproducing kernel. 
Together with a standard correction procedure, cf. \cite[Lem. 3, Thm.\ 11]{berlinet}, we derive an explicit 
formula given in Theorem \ref{repr_kernel} and \eqref{f7_2} below. The reproducing kernel is then being used 
to simulate the exact worst-case errors which represent the norm of the error functional, i.e. its Riesz representer, which 
can be computed exactly.
This approach allows to gain insights into the true behavior of the constants that are involved in the bounds for the integration error and usually only are estimated.
Let us emphasize once again that we simulate the worst-case error with respect to a whole function class rather than testing the algorithm on a single prototype test function.   

Finally, in Section 7 we show the results of several numerical experiments. In the first part of the experiment section 
we compare different well-known methods for numerical integration in the reproducing kernel Hilbert space framework 
which we established in Sections 5 and 6. In particular, we compare Frolov lattices based on different generating polynomials, 
the classical Frolov polynomials and the improved polynomials from Section 3. As one would expect, the numerical 
behavior of the respective worst-case errors differ significantly for small $\scalingFactor$. Where the improved polynomials lead to a 
rather satisfactory error decay, the classical method is numerically completely useless if the dimension increases. 
Interestingly, in case $d=2$ Frolov lattices according to the golden ratio polynomial compete with the Fibonacci lattice rule.
We also compare Frolov lattices and sparse grids with respect to the numerical performance. Note, that the 
sparse grid cubature method represents a further method which is able to benefit from higher (mixed) smoothness. 
However, it is well known \cite{DuUl15} that sparse grids show a worse behavior in the logarithm 
compared to Frolov lattices. Our experiments validate this theoretical fact. Among the considered methods (sparse grids, quasi-Monte Carlo) Frolov lattices
behave best in our setting. In addition, Frolov lattices do not have to be adapted to the present anisotropy when 
considering anisotropic mixed smoothness. When considering one minimal smoothness component \eqref{sm_vector} 
we observe the same pure polynomial rate in different dimensions, only the constant differs.
Note, that this effect would also be present for sparse grids 
adapted to the smoothness vector, which one has to know in advance. 

\textbf{Notation.} As usual $\N$ denotes the natural numbers, 
$\Z$ denotes the integers, 
and $\R$ the real numbers
. 
The letter $d$ is always reserved for the underlying dimension in $\R^d, \Z^d$ etc. We denote
with $\bsx \cdot \bsy$ 
the usual Euclidean inner product in $\R^d$. 
For $0<p\leq \infty$ we denote with $|\cdot |_p$ and $\|\cdot \|_p$ the ($d$-dimensional) 
discrete $\ell_p$-norm and the continuous $L_p$-norm on $\R$, respectively, 
where $B_p^d$ denotes the respective unit ball in $\R^d$.
The function $(\cdot)_+$ is given by $\max\{\cdot,0\}$.
With $\mathcal{F}$ we denote the Fourier transform given by $\mathcal{F} f(\boldsymbol{\xi}):=\int_{\R} f(x)\exp(-2\pi i \bsx\cdot \boldsymbol{\xi})\,\mathrm{d}\bsx$ for 
a function $f\in L_1(\R^d)$ and $\boldsymbol{\xi} \in \R^d$. For two sequences of real numbers $a_n$ and $b_n$ we will write 
$a_n \lesssim b_n$ if there exists a constant $c>0$ such that 
$a_n \leq c\,b_n$ for all $n$. We will write $a_n \asymp b_n$ if 
$a_n \lesssim b_n$ and $b_n \lesssim a_n$.  With $\mathrm{GL}_d:=\mathrm{GL}_d(\mathbb{R})$ we
denote the group of invertible matrices over $\R$, whereas $\mathrm{SO}_d:=\mathrm{SO}_d(\R)$ 
denotes the group of orthogonal matrices over $\R$ with unit determinant. With $\mathrm{SL}_d(\mathbb{Z})$ we 
denote the group of invertible matrices over $\Z$ with unit determinant.
The notation $D:=\mbox{diag}(x_1,...,x_d)$ with $\bsx = (x_1,...,x_d) \in \R^d$ 
refers to the diagonal matrix $D \in \R^{d\times d}$ with $\bsx$ at the diagonal.
With $\mathrm{gcd}(a,b)$ we denote the greatest common divisor of two positive integers $a,b$.
And finally, by $\mathbb{Z}[x]$ we denote the ring of polynomials with integer coefficients. 
Although we consider different generating matrices for admissible lattices
in the forthcoming, we do not specify the matrix in the denotation $Q^d_\scalingFactor$. 
This is, because we will fix, for every dimension $d$ under consideration, 
a matrix that is optimal in a sense that will be explained later.
To be precise, for a given dimension $d$, the matrix $A$ will be a multiple of 
the Vandermonde matrix as defined in Theorem~\ref{construction} with the 
specific polynomials (and roots) as given in Table~\ref{PolTable}.

\section{Admissible lattices and their representation} \label{sec:Admissible}

For a matrix $\bm{T}\in\GLdR$, we call $\{\bm{Tk}:\bm{k}\in \Z^d\} = \bm{T}(\Z^d)$
a (full-rank) lattice with lattice representation matrix $\bm{T}$.

For a matrix $\bm{U}\in\SLdZ$, the matrices $\bm{T}$ and $\bm{TU}$ generate the same lattice,
and it can easily be shown that all possible lattice representations of $\bm{T}(\Z^d)$ are given this way.
Therefore, it makes sense to define the determinant of a lattice $\bm{T}(\Z^d)$ as $|\det \bm{T}|$.
We want to mention that for a given lattice,
it is often preferred to have a lattice representation matrix $\bm{T} = (\bm{t}_1|\cdots|\bm{t}_d) \in \GLdR$
with column vectors $\bm{t}_1,\ldots,\bm{t}_d\in\R^d$ that are small with respect to some norm,
cf. Figure \ref{fig_equiv_lattice}.
\begin{figure}[t]
\centering
\includegraphics[width=0.48\linewidth]{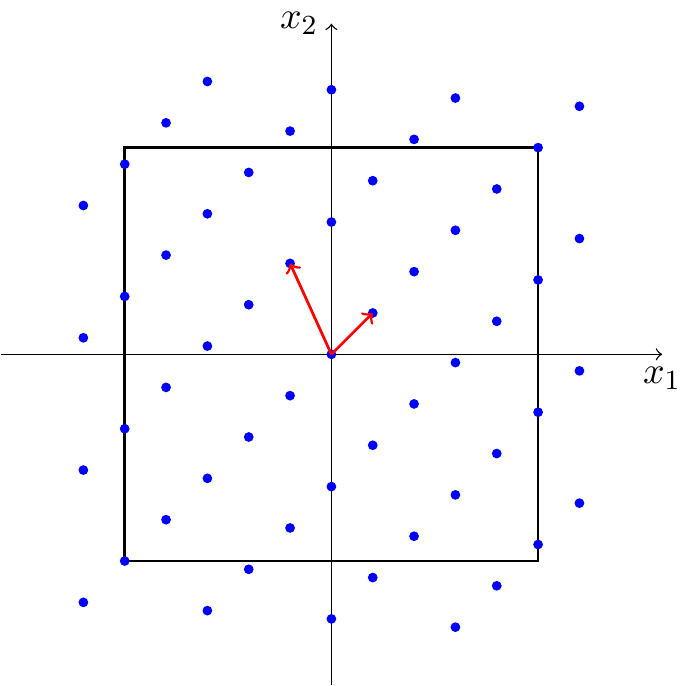}
\includegraphics[width=0.48\linewidth]{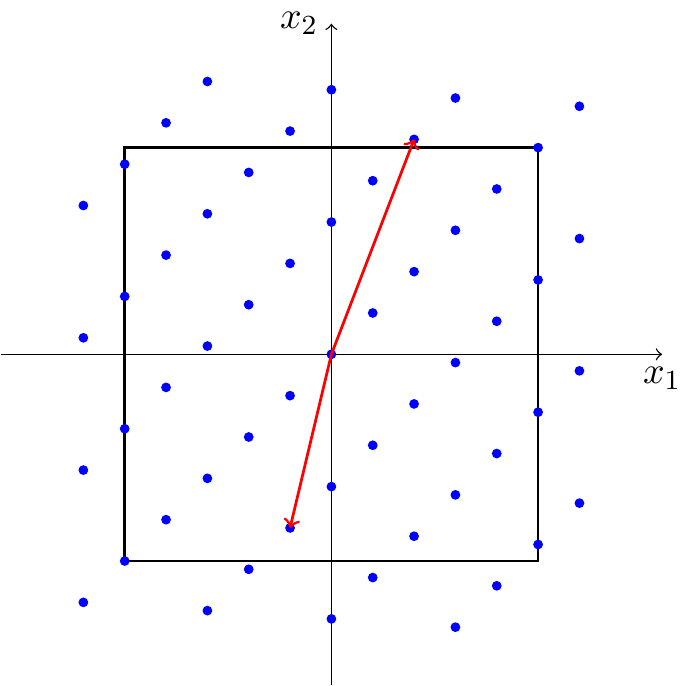}
\caption{Equivalent lattice representations within the unit cube $\Omega=\left[-1/2, 1/2\right]^2$. 
The highlighted lattice elements are the column vectors of the corresponding lattice representation matrix.}
\label{fig_equiv_lattice}
\end{figure}

Crucial for the performance of the Frolov cubature formula \eqref{eq:Q} will be the notion of \emph{admissibility}
which is settled in the following definition.
\begin{definition}[Admissible lattice]
A lattice $\bm{T}(\Z^d)$ is called admissible if
\begin{equation}\label{adm}
\Nm (\bm{T}) \coloneqq
 \inf_{\bm{k}\in\Z^d\setminus\{0\}} \Big|\prod_{i=1}^d (\bm{Tk})_i\Big| > 0
\end{equation}
holds true.
\end{definition}
\noindent Figure \ref{fig_lattice_hc} illustrates this property. 
In fact, lattice points different from $0$ lie outside of a hyperbolic
cross with 'radius' $\Nm(\bm{T})$.
\begin{figure}[t]
\centering
	\includegraphics[width=0.5\linewidth]{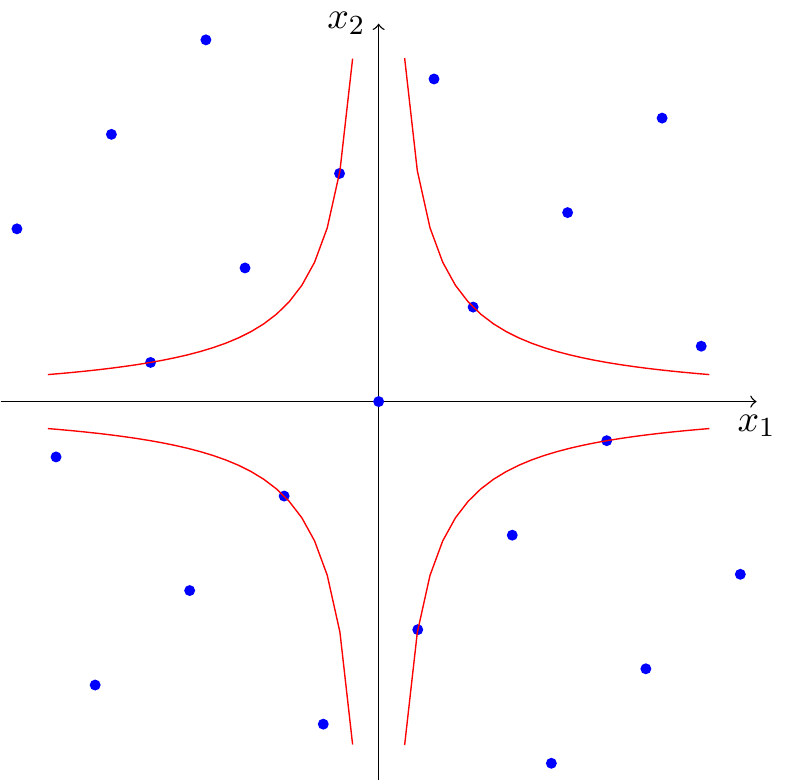}
	\caption{Admissible lattice and hyperbolic cross.} \label{fig_lattice_hc}
\end{figure}
Our construction of choice for admissible lattices is given by the following procedure.
\begin{proposition}\label{construction}
Let $P(x)$ be a polynomial of degree $d$ satisfying
\begin{itemize}
\item $P$ has integer coefficients,
\item $P$ has leading coefficient $1$,
\item $P$ is irreducible over $\Q$,
\item $P$ has $d$ different real roots $\xi_1, \ldots , \xi_d$\,.
\end{itemize}
The Vandermonde matrix 
\begin{equation}\label{Vanderm}
    \bm{V} = \left(\begin{array}{cccc}
                1 & \xi_1 & \cdots & \xi_1^{d-1}\\
                1 & \xi_2 & \cdots & \xi_2^{d-1}\\
                \vdots & \vdots & \ddots & \vdots\\
                1 & \xi_{d} & \cdots & \xi_d^{d-1}
              \end{array}\right)\,\in\GLdR
\end{equation}
generates an admissible lattice $\bm{V}(\Z^d)$ with $\Nm(\bm{V}) = 1$.
Its determinant equals the polynomial discriminant $D_P$ of $P$:
\begin{equation}\label{det_Vandermonde}
|\det \bm{V}| = \prod_{k<l}|\xi_k-\xi_l| = D_P\,.
\end{equation}
Moreover, it holds
\begin{equation}\label{DualityOfNorm}
\Nm(\bm{V}^{-\top}) = |\det \bm{V}|^{-2} = D_P^{-2}\,.
\end{equation}
\end{proposition}



The necessary prequisites of $P$ can be reformulated with concepts of algebraic number theory:
$P$ is the minimal polynomial of an algebraic integer of order $d$.
For the proof of this statement we refer to \cite{Kac16}, 
or \cite{GrLekk87} and \cite{Marcus} for a thorough introduction into the theory of algebraic integers.
The quantity \eqref{DualityOfNorm}
has a direct impact on the convergence behavior of the Frolov cubature formula
and we therefore are interested in polynomials which maximize this quantity for a fixed $d$,
i.e. have a small (or the smallest) polynomial discriminant $D_P$, cf. Figure \ref{2dLattices}.

\begin{figure}[t]
\centering
\begin{subfigure}[b]{0.22\textwidth}
\centering
\includegraphics[width=\textwidth]{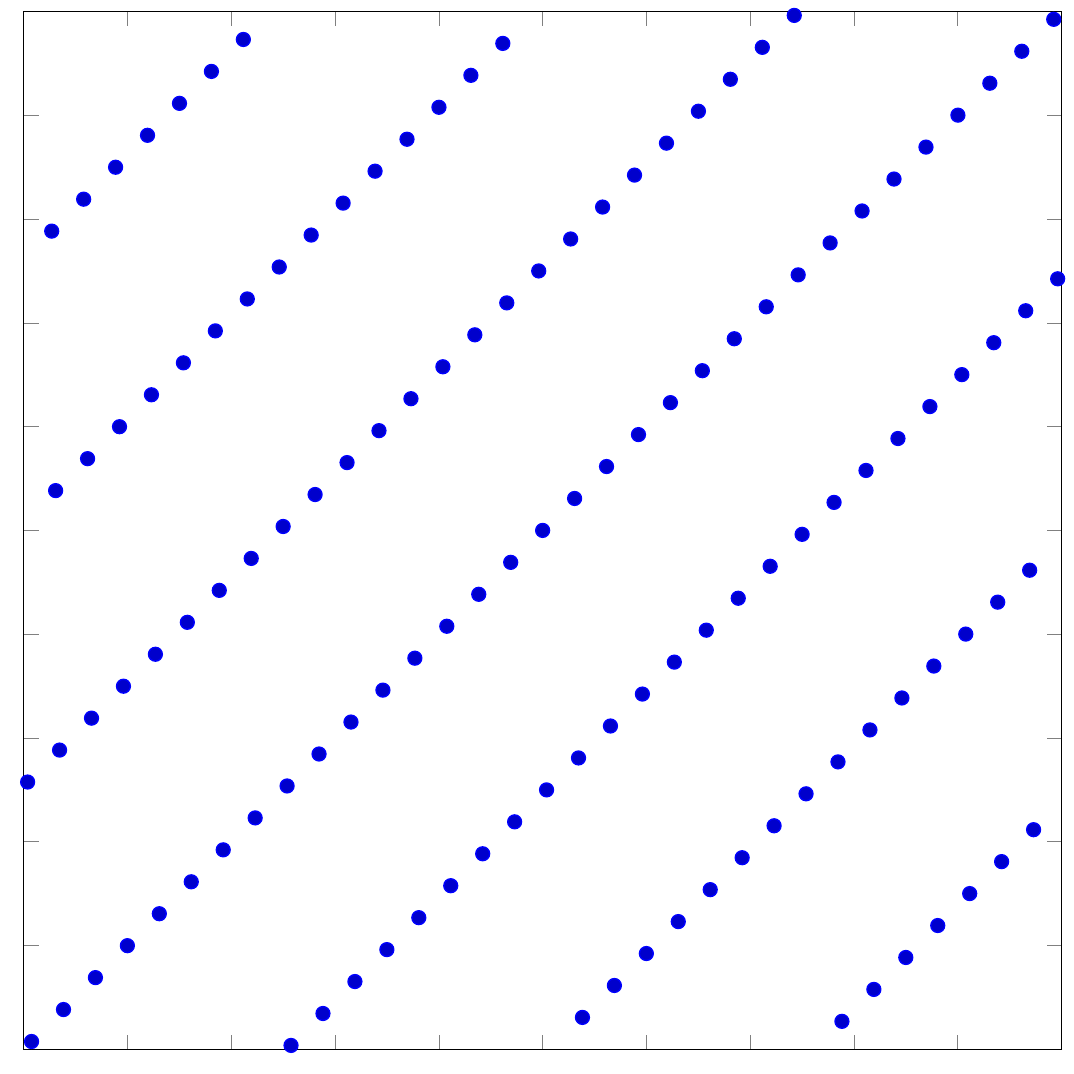}
\caption{ \small{$P(x) = x^2-17$, \\$D_P = 2\sqrt{17}$}}
\end{subfigure}
\begin{subfigure}[b]{0.22\textwidth}
\centering
\includegraphics[width=\textwidth]{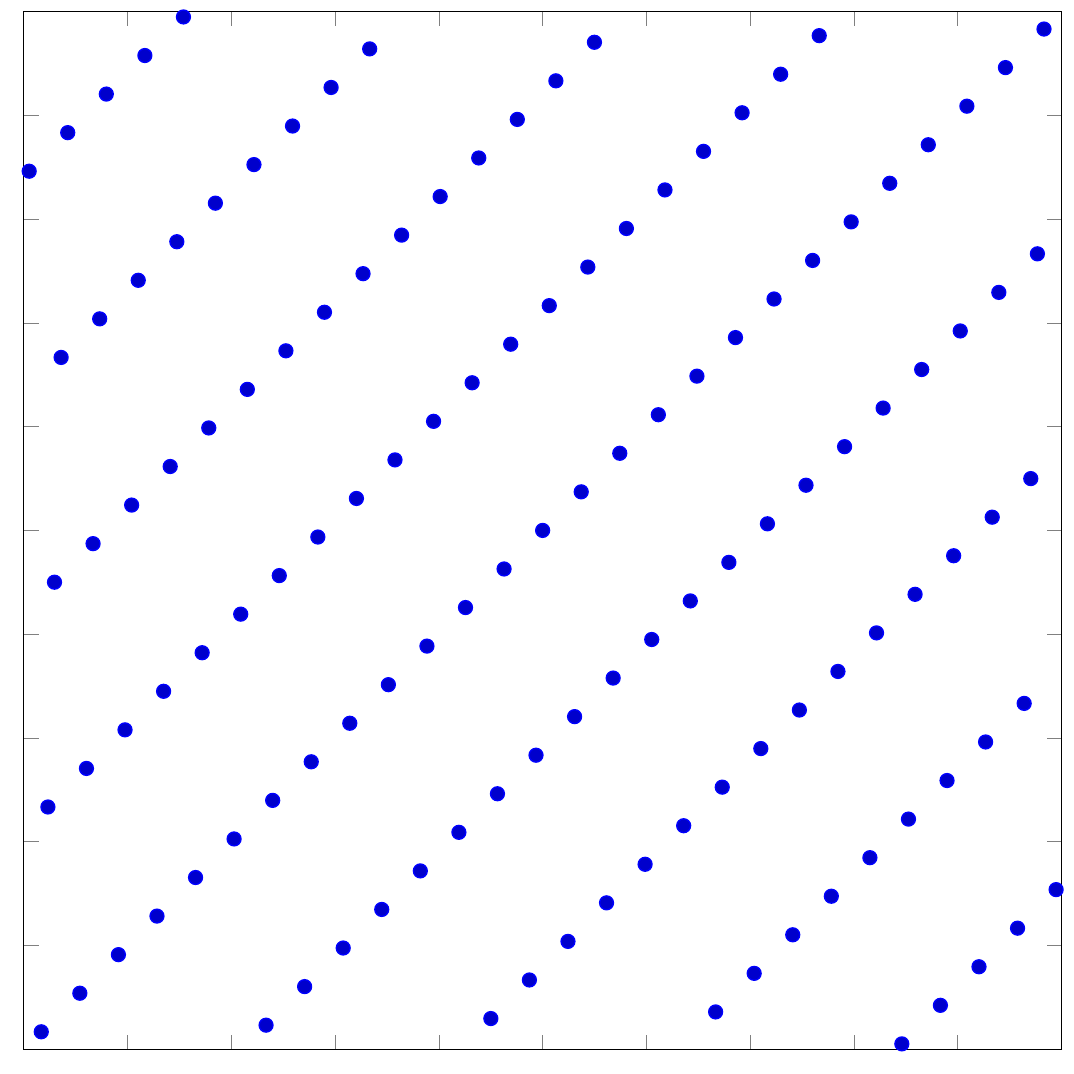}
\caption{ \small{$P(x) = x^2-8$, \\$D_P = 2\sqrt{8}$}}
\end{subfigure}
\begin{subfigure}[b]{0.22\textwidth}
\centering
\includegraphics[width=\textwidth]{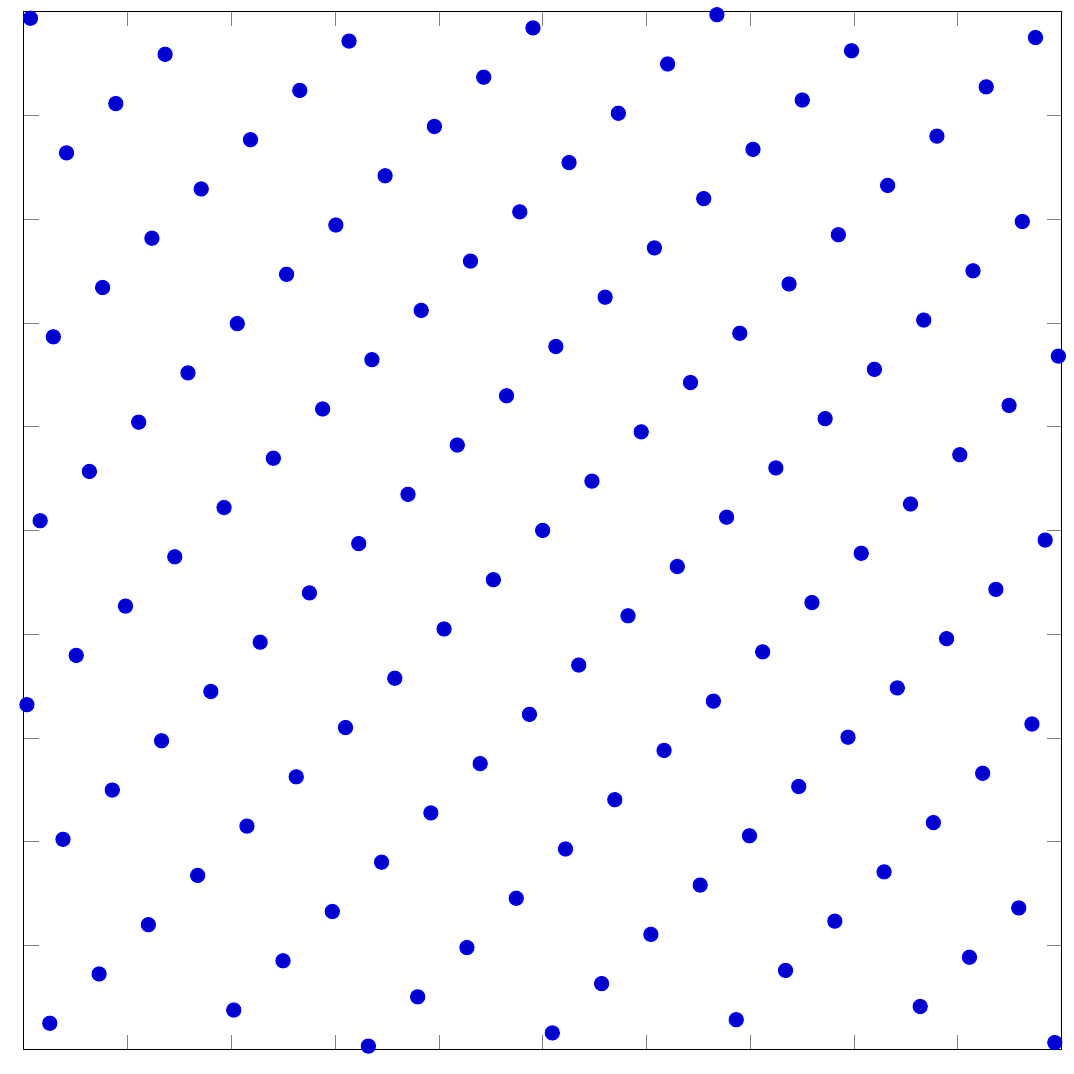}
\caption{ \small{$P(x) = x^2-3$, \\$D_P = 2\sqrt{3}$}}
\end{subfigure}
\begin{subfigure}[b]{0.22\textwidth}
\centering
\includegraphics[width=\textwidth]{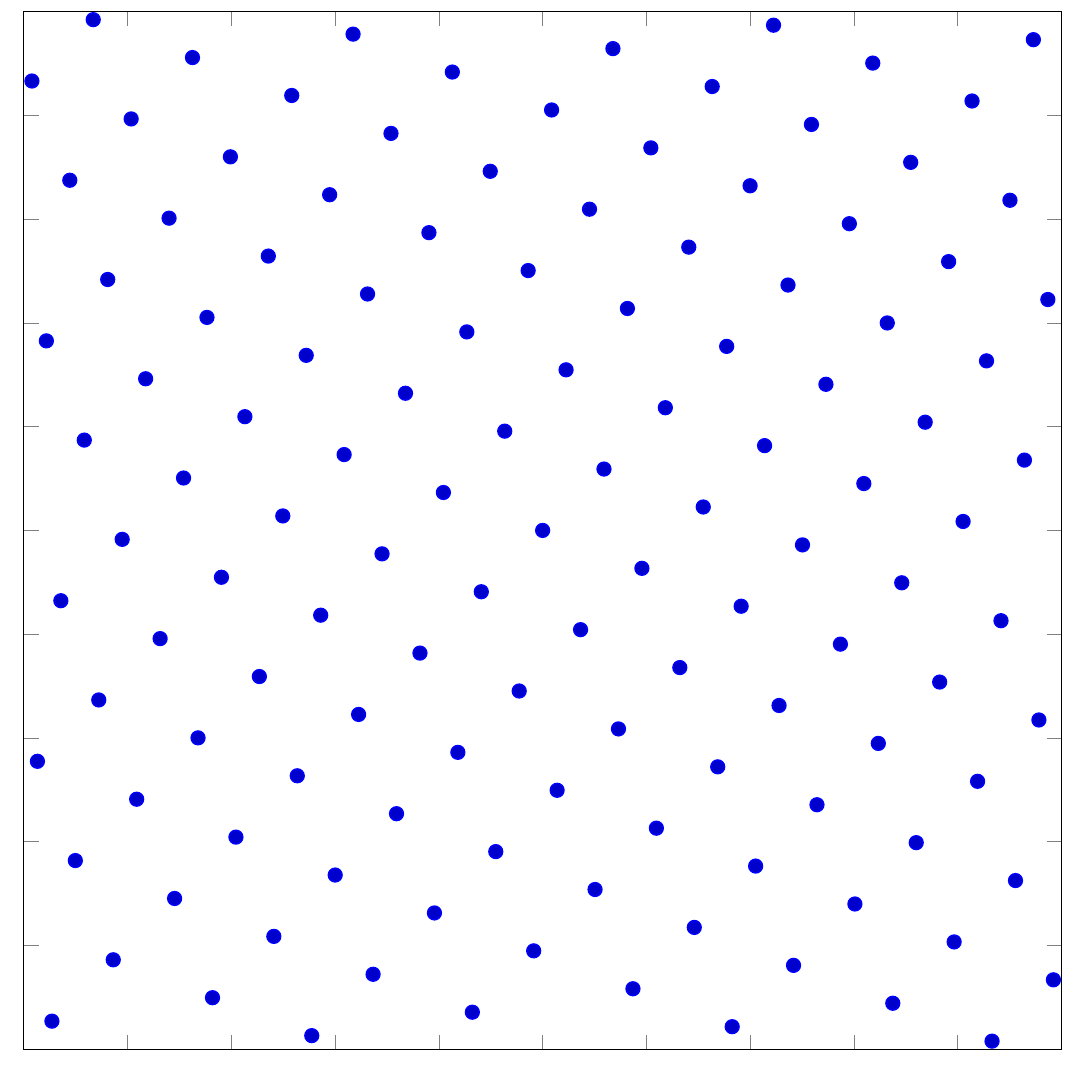}
\caption{ \small{$P(x) = x^2-x-1$, \\$D_P = \sqrt{5}$}}
\end{subfigure}
\caption{Lattices corresponding to different polynomials for $d=2$.
A small discriminant correlates with a good distribution of lattice points.}
\label{2dLattices}
\end{figure}

Using Proposition \ref{construction}, we obtain the lattice $\bm{V}(\Z^d)$ represented by a Vandermonde matrix $\bm{V}$.
There are two problems with such matrices from the numerical point of view: 
First, they have large column vectors and therefore a large condition number, 
and second, its entries are of the form $\bm{V}_{ij} = \xi_i^{j-1}$
for which the calculation gets unstable for increasing $j$. 
However, we can bypass this problem using special polynomials, which will be discussed in the next section.
\begin{lemma}\label{LatRep}
Let $P$ be a polynomial which satisfies the prequisites in Proposition \ref{construction},
and has roots $\xi_1,\ldots,\xi_d$ which lie in $(-2, 2)$.
Furthermore, let $\omega_1,\ldots,\omega_d\in (-1,1)$ be defined via the equation
\begin{equation*}
2\cos(\pi\omega_k) = \xi_k\,,\quad k=1,\ldots,d\,.
\end{equation*}
The lattice $\bm{V}(\Z^d)$ generated by the associated Vandermonde matrix
\begin{equation*}
    \bm{V} = \left(\begin{array}{cccc}
                1 & \xi_1 & \cdots & \xi_1^{d-1}\\
                1 & \xi_2 & \cdots & \xi_2^{d-1}\\
                \vdots & \vdots & \ddots & \vdots\\
                1 & \xi_{d} & \cdots & \xi_d^{d-1}
              \end{array}\right)\,\in\GLdR
\end{equation*}
is also generated by the matrix $\bm{T}$ with
\begin{equation*}
\bm{T}_{kl} =
\begin{cases}
	1&
    l=1\,,\\
	2\cos\left(\pi (l-1)\omega_k\right)&
    l=2,\dots, d\,.
\end{cases}
\end{equation*}
\end{lemma}
The resulting matrix $\bm{T}$ has entries in $(-2,2)$ that can be calculated in a numerically stable way,
which is optimal for our purposes.
The proof is a straightforward application of Euler's identity
and can be found in \cite{Kac16, KaOeUl17}.

Finally,
we specify our choice of the matrices $\bm{A}$ and $\bm{A}_\scalingFactor$ used in Frolov's cubature formula \eqref{eq:Q} as
\begin{equation}
\bm{A} = |\det \bm{T}|^{-1/d}\bm{T}, \quad \bm{A}_\scalingFactor = \scalingFactor^{-1/d}\bm{A}
\end{equation}
with $\bm{T}$ as in Lemma \ref{LatRep}.

\section{Improved generating polynomials} \label{sec:polynomials}
In this section, we consider polynomials which can be used to create admissible lattices.
We will call such polynomials \emph{admissible}, 
i.e. a $d$-th order polynomial $P$ is admissible if it satisfies the prequisites of Proposition \ref{construction}.
At the end of this section,
we provide a list of admissible polynomials with small discriminant for $d=2,\ldots,10$.

The study of Chebyshev Polynomials of the first and second kind 
provides us with a wide range of admissible polynomials. 
The most important features are their real and pairwise different roots, 
as well as the narrow distribution thereof. 
It is also quite fortunate to us that the decomposition into irreducible factors is well-understood 
and can be stated explicitly, see \cite{RTW98}.
\begin{definition}
The Chebyshev Polynomials of the first kind $T_d(x)$ are defined recursively via
\begin{eqnarray*}
T_0(x) &=& 1,\\ T_1(x) &=& x,\\ T_d(x) &=& 2xT_{d-1}(x) - T_{d-2}(x),\quad d\geq 2\,.
\end{eqnarray*}
The Chebyshev Polynomials of the second kind $U_d(x)$ are defined recursively via
\begin{eqnarray*}
U_0(x) &=& 1,\\ U_1(x) &=& 2x,\\ U_d(x) &=& 2xU_{d-1}(x) - U_{d-2}(x),\quad d\geq 2\,.
\end{eqnarray*}
\end{definition}
\begin{lemma}
The Chebyshev polynomial $T_d(x)$ has the roots
\begin{equation*}
\cos\left(\frac{\pi (2k-1)}{2d}\right),\quad k=1,\dots, d\,.
\end{equation*}
The Chebyshev polynomial $U_d(x)$ has the roots
\begin{equation*}
\cos\left(\frac{\pi k}{d+1}\right),\quad k=1,\dots, d\,.
\end{equation*}
\end{lemma}
The polynomials $T_d(x)$ and $U_d(x)$ are not admissible since they do not have leading coefficient $1$.
But they can be scaled appropriately to achieve this.
\begin{lemma}
The scaled Chebyshev polynomials ${\widetilde{T}}_d(x) = 2T_d(x/2)$
and ${\widetilde{U}}_d(x) = U_d(x/2)$ have leading coefficient $1$ and belong to $\Z[x]$.
The scaled Chebyshev polynomial ${\widetilde{T}}_d(x)$ has the roots
\begin{equation*}
t_{d,k} = 2\cos\left(\frac{\pi (2k-1)}{2d}\right),\quad k=1,\dots, d\,.
\end{equation*}
The scaled Chebyshev polynomial ${\widetilde{U}}_d(x)$ has the roots
\begin{equation*}
u_{d,k} = 2\cos\left(\frac{\pi k}{d+1}\right),\quad k=1,\dots, d\,.
\end{equation*}
\end{lemma}
>From this lemma it follows directly that irreducible factors of ${\widetilde{T}}_d(x)$
and ${\widetilde{U}}_d(x)$ are admissible and have roots that lie in $(-2,2)$.
As already stated above, \cite{RTW98} lists the complete decomposition of Chebyshev Polynomials into irreducible factors, 
which we reformulate for the scaled versions in the next lemma.
\begin{lemma}\label{factorization}
For a fixed $d>1$, we have 
\begin{equation*}
{\widetilde{T}}_d(x) = \prod_h D_{d,h}(x)\,,
\end{equation*}
where $h\leq d$ runs through all odd positive divisors of $d$ and
\begin{equation*}
D_{d,h}(x) = \prod_{\substack{k=1 \\ \gcd(2k-1,d)=h}}^d(x-t_{d,k})
\end{equation*}
are all irreducible. It also holds 
\begin{equation*}
{\widetilde{U}}_d(x) = \prod_h E_{d,h}(x)\,,
\end{equation*}
where $h\leq d$ runs through all positive divisors of $2d+2$ and
\begin{equation*}
E_{d,h}(x) = \prod_{\substack{k=1\\\gcd(k,2d+2)=h}}^d(x-u_{d,k})
\end{equation*}
are all irreducible.
\end{lemma}
It has been shown in \cite{tem93} that ${\widetilde{T}}_d(x)$ is irreducible for $d=2^m, m\in\N$
and that the corresponding lattice is orthogonal \cite{KaOeUl17}.
However, in this paper we are more interested in the irreducible factors of ${\widetilde{U}}_d(x)$, 
mainly for two reasons. 
First, it can be easily seen that the irreducible factors of ${\widetilde{T}}_d(x)$ have paired roots, i.e.
\begin{equation*}
D_{d,h}(t_{d,k}) = 0 \Rightarrow D_{d,h}(t_{d,d-k+1}) = 0\,.
\end{equation*}
This means that either $D_{d,h}(x) = x$ or $D_{d,h}(x)$ is a polynomial of even degree, 
limiting the usefulness to our purposes. 
Second, it appears to be the case that the discriminant of ${\widetilde{U}}_d(x)$
is smaller than the discriminant of ${\widetilde{T}}_d(x)$,
which makes the factors of ${\widetilde{U}}_d(x)$ more attractive to us.
The following lemma is a consequence of Lemma \ref{factorization}.

\begin{sidewaystable}
\centering
\begin{tabular}{|c|c|l|c|}
\hline
dimension $d$ & notation & polynomial \& roots & discriminant $D_P$\\
\hline
\multirow{2}{*}{2} & \multirow{2}{*}{$E_{4,2}(x)$} & $x^2+x-1$ & \multirow{2}{*}{$2.24$} \\
&&$2\cos\left(\pi \frac{2i}{2d+1} \right),\, i=1,\ldots,d$&\\
\hline
\multirow{2}{*}{3} & \multirow{2}{*}{$E_{6,2}(x)$} & $x^3+x^2-2x-1$ & \multirow{2}{*}{$7$} \\
&&$2\cos\left(\pi \frac{2i}{2d+1} \right),\, i=1,\ldots,d$&\\
\hline
\multirow{2}{*}{4} & \multirow{2}{*}{$E_{14,2}(x)$} & $x^4-x^3-4x^2+4x+1$ & \multirow{2}{*}{$33.54$} \\
&&$2\cos\left(\pi \frac{2}{15}\right),2\cos\left(\pi \frac{4}{15}\right),2\cos\left(\pi \frac{8}{15}\right),2\cos\left(\pi \frac{14}{15}\right)$&\\
\hline
\multirow{2}{*}{5} & \multirow{2}{*}{$E_{10,2}(x)$} & $x^5+x^4-4 x^3-3 x^2+3 x+1$ & \multirow{2}{*}{$121$} \\
&&$2\cos\left(\pi \frac{2i}{2d+1} \right),\, i=1,\ldots,d$&\\
\hline
\multirow{2}{*}{6} & \multirow{2}{*}{$E_{12,2}(x)$} & $x^6+x^5-5 x^4-4 x^3+6 x^2+3 x-1$ & \multirow{2}{*}{$609.34$} \\
&&$2\cos\left(\pi \frac{2i}{2d+1} \right),\, i=1,\ldots,d$&\\
\hline
\multirow{2}{*}{7} & \multirow{2}{*}{$P_7(x)$} & $x^7+x^6-6x^5-4x^4+10x^3+4x^2-4x-1$ & \multirow{2}{*}{$4487.14$} \\
&&no explicit formula available&\\
\hline
\multirow{2}{*}{8} & \multirow{2}{*}{$E_{16,2}(x)$} & $x^8+x^7-7 x^6-6 x^5+15 x^4+10 x^3-10 x^2-4 x+1$ & \multirow{2}{*}{$20256.8$} \\
&&$2\cos\left(\pi \frac{2i}{2d+1} \right),\, i=1,\ldots,d$&\\
\hline
\multirow{2}{*}{9} & \multirow{2}{*}{$E_{18,2}(x)$} & $x^9+x^8-8x^7-7x^6+21x^5+15x^4-20x^3-10x^2+5x+1$ & \multirow{2}{*}{$130321$} \\
&&$2\cos\left(\pi \frac{2i}{2d+1} \right),\, i=1,\ldots,d$&\\
\hline
\multirow{3}{*}{10} & \multirow{3}{*}{$E_{24,2}(x)$} & $x^{10}-10x^8+35x^6+x^5-50x^4-5x^3+25x^2+5x-1$ & \multirow{3}{*}{$873464$} \\
&&$2\cos\left(\pi \frac{2}{25}\right),2\cos\left(\pi \frac{4}{25}\right),2\cos\left(\pi \frac{6}{25}\right),2\cos\left(\pi \frac{8}{25}\right),2\cos\left(\pi \frac{12}{25}\right),$&\\
&&$2\cos\left(\pi \frac{14}{25}\right),2\cos\left(\pi \frac{16}{25}\right),2\cos\left(\pi \frac{18}{25}\right),2\cos\left(\pi \frac{22}{25}\right),2\cos\left(\pi \frac{24}{25}\right)$&\\
\hline
\end{tabular}
\caption{Admissible polynomials with small discriminants for $d=2,\ldots, 10$.}
\label{PolTable}
\end{sidewaystable}

\begin{lemma}
Let $d>1$. If $p=2d+1$ is a prime, the $d$th-order polynomial
\begin{equation*}
E_{2d,2}(x) = \prod_{k=1}^d (x-u_{2d,2k})
\end{equation*}
is admissible.
\end{lemma}
\begin{proof}
Consider the factorization of ${\widetilde{U}}_{2d}(x)$. We have $2(2d)+2=4d+2=2p$, 
therefore we have for $1\leq k\leq 2d$
\begin{equation*}
\gcd(k,2(2d)+2) = 
\begin{cases}
1 & k\text{ odd} \\
2 & k\text{ even} \,.
\end{cases}
\end{equation*}
This implies that ${\widetilde{U}}_{2d}(x)=E_{2d,1}(x)E_{2d,2}(x)$, which both are of order $d$.  
\end{proof}
This simple rule covers the cases $d\in\{2,3,5,6,8,9\}$. 
For the cases $d=4$ and $d=10$ we also did find good factors. 
\begin{lemma}
The polynomial $E_{14,2}(x)$ is of order $4$ and admissible 
and the polynomial $E_{24,2}(x)$ is of order $10$ and admissible.
\end{lemma}
\begin{proof}
    Both polynomials are admissible by definition, it remains to compute their order.
    We first consider $E_{14,2}(x)$.
    Here, $2\cdot 14 +2 = 30$, and for $1\leq k \leq 14$ one has $\gcd(k, 30) = 2$
    if and only if $k\in\{2,4,8,14\}$.
    Therefore, $E_{14,2}(x)$ is a polynomial of order $4$.
    Now consider $E_{24,2}(x)$.
    We have $2\cdot 24 + 2 = 50$, and for $1\leq k \leq 24$ one has $\gcd(k,50) = 2$
    if and only if $k\in\{2,4,6,8,12,14,16,18,22,24\}$.
    Therefore, $E_{24,2}(x)$ is a polynomial of order $10$.
\end{proof}
Unfortunately, the case $d=7$ is not covered by the factorization of all ${\widetilde{U}}_d(x)$ and ${\widetilde{T}}_d(x)$.
However, using a numerical brute force approach, we found the following polynomial.
\begin{lemma}
The polynomial 
\begin{equation*}
P_7(x) = x^7 + x^6 - 6x^5 - 4x^4 + 10x^3 + 4x^2 - 4x - 1
\end{equation*}
is of order 7 and admissible.
\end{lemma}
\begin{proof}
    We have to prove that $P_7(x)$ is irreducible over $\Q$.
    It has leading coefficient $1$ and coefficients in $\Z$,
    therefore it is irreducible over $\Q$ if and only if it is irreducible over $\Z$.
    Here, we consider irreducibility over $\mathbb{F}_2$,
    which is a sufficient condition for irreducibility over $\Z$.
    In $\mathbb{F}_2$, one has
    \begin{equation*}
        P_7(x) \equiv x^7 + x^6 + 1\,.
    \end{equation*}
    Assume that this polynomial is reducible.
    Because it has no roots in $\mathbb{F}_2$,
    it would have to contain a factor of degree less then $4$ which also has no root in $\mathbb{F}_2$.
    The possible candidates are therefore $x^2+x+1$, $x^3+x+1$ and $x^3+x^2+1$.
    Doing a polynomial division with these three polynomials,
    one finds that
    \begin{align*}
        x^7+ x^6 + 1 &\equiv (x^2+x+1)(x^5+x^3+x^2+1) &+ x \\
        &\equiv (x^3+x+1)(x^4+x^3+x^2) &+ x^2 + 1 \\
        &\equiv (x^3+x^2+1)(x^4+x+1) &+ x^2 + x
    \end{align*}
    and we have a contradiction.
    Therefore, $P_7(x)$ is irreducible over $\mathbb{F}_2$, and subsequently also over $\Q$.
\end{proof}
Even though Lemma \ref{LatRep} is not applicable for this polynomial because its roots lie in $(-2.25, 1.75)$,
they still lie close to each other,
which results in a good polynomial discriminant.
Regarding the lattice representation issue in the $d=7$ case,
one has to compute the Vandermonde matrix $\bm{V}$
explicitly (using an arbitrary precision data type to avoid stability issues)
and find a good lattice representation matrix $\bm{T}$
by means of a lattice reduction algorithm, see for instance \cite{Lenstra1982}.

This completes our list of polynomials used for the dimensions $2\leq d\leq 10$.
We attach Table \ref{PolTable} collecting all polynomials and useful information.

\section{Efficient enumeration of Frolov lattices in $d$-cubes} \label{sec:enumeration}

\SetKwFunction{assemble}{assemble}
\SetKwProg{Fn}{Function}{}{}
\begin{algorithm}[t] 
    \textbf{Input:}\\ 
    Integration domain $\Omega = [-1/2, 1/2]^d$,\\
    Lattice representation matrix $\bm{T}=\bm{QR}$\\
	\hrulefill\\
	\textbf{set} $\mathcal{N} = \emptyset$\\
	\textbf{set} $\bm{m} = (0,\ldots,0)^\top$\\
	\textbf{run} \assemble $(\mathcal{N}, d, \bm{m})$\\
	\hrulefill\\
    \Fn{\assemble $(\mathcal{N}, j, \bm{m})$}{
		\If{$j\geq 2$}{
			Determine the set $K_j =
			\{{k}_j \in\Z\,\colon\,g_j(0,\ldots,0,{k}_j,{m}_{j+1},\ldots,{m}_d)
			\leq \frac{d}{4} - \sum_{i=j+1}^d g_i(\bm{m}) \}$\\
			\ForAll{$k_j\in K_j$}{
				\textbf{set} $m_j = k_j$\\
				\assemble $(\mathcal{N}, j-1, \bm{m})$
			}
			\textbf{set} $m_j = 0$\\
		}
		\If{$j=1$}{
			Determine the set
			$K_1 = \{k_1\in\Z\,\colon\,g_j(k_1,m_{2},\ldots,m_d)
			\leq \frac{d}{4} - \sum_{i=2}^d g_i(\bm{m}) \}$\\
			\ForAll{$k_1\in K_1$}{
				\textbf{set} $m_1 = k_1$\\
				\If{$\bm{Tm}\in\Omega$}{
					\textbf{set} $\mathcal{N} = \mathcal{N}\cup \{\bm{Tm}\}$
				}
			}
			\textbf{set} $m_1 = 0$\\
		}
    }
    \hrulefill\\
    \textbf{Output}: Set of lattice points $\mathcal{N}$ 
	\caption{Assemblation of the set $\mathcal{N} = \Omega\cap\bm{T}(\Z^d)$.} \label{QRassemble}
\end{algorithm}

In this section we present an enumeration algorithm
to determine the set of integration points for the Frolov cubature formula.
The approach is similar to the one in \cite{KaOeUl17} for orthogonal lattices,
Here, we generalize the method for arbitrary lattices.

\subsection{Enumeration of non-orthogonal Frolov lattices}
We fix the integration domain $\Omega = [-1/2, 1/2]^d$ and a lattice $\bm{T}(\Z^d)$
with lattice representation matrix $\bm{T}$.
We are interested in the discrete set 
\begin{equation*}
\mathcal{N} = \Omega\cap\bm{T}(\Z^d) = \{\bm{Tk}\in\Omega \,\colon \, \bm{k}\in\Z^d\}\,.
\end{equation*}
Our strategy is to consider a slightly larger set $\mathcal{B} \supset \mathcal{N}$ 
which allows for explicit enumeration in a straightforward way. 
We choose 
\begin{equation*}
\mathcal{B} = B_{\sqrt{d}/2}(0)\cap\bm{T}(\Z^d)
= \left\{\bm{Tk} \,\colon\, \|\bm{Tk}\|_2^2\leq \frac{d}{4}, \bm{k}\in\Z^d\right\}\,.
\end{equation*}
Using the matrix decomposition 
\begin{equation*}
\bm{T} = \bm{QR}\,,
\end{equation*}
where $\bm{Q}$ is an orthogonal matrix and $\bm{R}$ is an upper triangular matrix,
we can rewrite this set as
\begin{equation*}
\mathcal{B} = \left\{\bm{Tk} \,\colon\, \|\bm{Rk}\|_2^2\leq \frac{d}{4}, \bm{k}\in\Z^d\right\}\,.
\end{equation*}
The function $\|\bm{R}\cdot\|_2^2$ can be split up into additive parts
\begin{eqnarray*}
\|\bm{Rk}\|_2^2 &=& \sum_{i=1}^d  g_i(\bm{k})\\
g_i(\bm{k}) &=& (\bm{Rk})_i^2,\quad i=1,\ldots,d\,,
\end{eqnarray*}
and from the upper triangular structure of $\bm{R}$ it follows that
$g_j(\bm{k})$ only depends on the components $k_j,\ldots,k_d$.
For an integer vector $\bm{k}$ we therefore have
\begin{equation}
\|\bm{Rk}\|_2^2 \leq  \frac{d}{4}
\Longleftrightarrow  g_j(\bm{k}) \leq \frac{d}{4} - \sum_{i=j+1}^d g_i(\bm{k})\,,\quad j=1,\ldots,d\,.
\end{equation}
Fixing the coordinates $k_{j+1},\ldots,k_d$ results in explicitly solvable inequalities for $k_j$,
since the right hand side is constant and the left hand side is a quadratic function in $k_j$.
Therefore, the set $\mathcal{N}$ can be assembled with Algorithm \ref{QRassemble}.

This algorithm iterates over all elements of $\mathcal{B}$, 
which determines the complexity that is of order 
\begin{equation*}
\vol_d\left(B_{\sqrt{d}/2}(0)\right)/|\det \bm{T}| \asymp 2^d \cdot |\mathcal{N}|\,.
\end{equation*}
This is certainly true if the sets $K_j$ appearing in the algorithm are all nonempty, 
and this should be the case for a lattice with a small determinant and a good choice of its representation matrix. 
The exponential dependence on $d$ is of minor importance here;
Once the Frolov integration points are computed and stored,
they can be reused for numerical integration.

\subsection{Numerical results}
In Table \ref{tab_times} the running times for the enumeration of the Frolov lattice points in $[0,1]^d$ with Algorithm \ref{QRassemble} are provided for dimensionalities $d \in \{2,3, \ldots, 9\}$. 
Firstly, we observe that the number of points $N$ converges to the scaling factor $n$, as $n$ becomes large, cf. \eqref{eqn_scalingVolume}.

Moreover, one can observe the \emph{linear runtime} of the algorithm in terms of the number of points $N$: If the number of points $N$ is quadrupled, then also the required time to assemble these $4N$ points is approximatively quadrupled.
However, comparing the runtimes for small $d$ and large $d$, it is apparent that a dimension-dependent constant is involved. 
This is analogous to the orthogonal setting  for $d = 2^k, k \in \N$, as it was treated in \cite{KaOeUl17}.

The resulting point sets for dimension $d \in \{2,3,\ldots, 10\}$ are available for download at \\
\texttt{http://wissrech.ins.uni-bonn.de/research/software/frolov/}.

\begin{table}[t]
 \begin{tabular}{|l|l|l|l|}
\hline
 Dim. $d$ & Scaling $n$ & Points $N$& Time  (s)\\
\hline
2 & 1024 & 1023 & 4.4e-05\\
2 & 4096 & 4093 & 0.000158\\
2 & 16384 & 16387 & 0.00053\\
2 & 65536 & 65533 & 0.002117\\
2 & 262144 & 262147 & 0.00823\\
2 & 1048576 & 1048575 & 0.096369\\
\hline
3 & 1024 & 1021 & 0.000105\\
3 & 4096 & 4093 & 0.000341\\
3 & 16384 & 16387 & 0.001213\\
3 & 65536 & 65537 & 0.004547\\
3 & 262144 & 262149 & 0.017474\\
3 & 1048576 & 1048581 & 0.114605\\
\hline
4 & 1024 & 1023 & 0.00024\\
4 & 4096 & 4103 & 0.000805\\
4 & 16384 & 16395 & 0.002844\\
4 & 65536 & 65551 & 0.010464\\
4 & 262144 & 262155 & 0.038923\\
4 & 1048576 & 1048579 & 0.248508\\
\hline
5 & 1024 & 1021 & 0.00061\\
5 & 4096 & 4093 & 0.002072\\
5 & 16384 & 16359 & 0.007013\\
5 & 65536 & 65533 & 0.025019\\
5 & 262144 & 262141 & 0.129366\\
5 & 1048576 & 1048591 & 0.473579\\

\hline
 \end{tabular}
 \quad \, \quad
\begin{tabular}{|l|l|l|l|}
\hline
 Dim. $d$ & Scaling $n$ & Points $N$& Time (s) \\
\hline
6 & 1024 & 1005 & 0.00146\\
6 & 4096 & 4087 & 0.004961\\
6 & 16384 & 16401 & 0.016533\\
6 & 65536 & 65513 & 0.059226\\
6 & 262144 & 262161 & 0.241978\\
6 & 1048576 & 1048585 & 0.943112\\
\hline
7 & 1024 & 1009 & 0.003208\\
7 & 4096 & 4099 & 0.011418\\
7 & 16384 & 16383 & 0.039014\\
7 & 65536 & 65531 & 0.13972\\
7 & 262144 & 262117 & 0.513067\\
7 & 1048576 & 1048573 & 2.0007\\
\hline
8 & 1024 & 1029 & 0.007961\\
8 & 4096 & 4051 & 0.025833\\
8 & 16384 & 16441 & 0.094269\\
8 & 65536 & 65539 & 0.329561\\
8 & 262144 & 262207 & 1.20636\\
8 & 1048576 & 1048767 & 4.59066\\
\hline
9 & 1024 & 997 & 0.017742\\
9 & 4096 & 4035 & 0.066017\\
9 & 16384 & 16517 & 0.223132\\
9 & 65536 & 65557 & 0.76848\\
9 & 262144 & 262107 & 2.77068\\
9 & 1048576 & 1048631 & 10.4136\\
\hline

\end{tabular}
\caption{Running times for the assemblation of Frolov cubature points in $[0,1]^d$.} \label{tab_times}

\end{table}

\section{Compactly supported functions with bounded mixed derivative in $L_2$} \label{sec:space}

\subsection{Characterization of the space}
We denote with $S(\R^d)$ the usual Schwartz space. Let $\mathbf{r} = (r_1,...,r_d) \in \N^d$ be a smoothness vector 
with integer components. Then we define the semi-norm
$$
  \|\varphi\|^2_{H^{\mathbf{r}}_{\text{mix}}} := \sum\limits_{e \subset [d]}
  \Big\|\Big(\prod\limits_{i\in e} \frac{\partial^{r_i}}{\partial x_i^{r_i}}\Big)\varphi\Big\|^2_2 
$$
where $\|\cdot\|_2$ denotes the $L_2(\R^d)$-norm. Clearly this norm is induced by an inner product. 
By Plancherel's theorem together with well-known properties of the Fourier transform, see \eqref{FT} below,
we may rewrite  
\begin{equation}\label{potential}
\begin{split}
  \|\varphi\|_{H^{\mathbf{r}}_{\text{mix}}} &= \Big\|\mathcal{F}^{-1}\Big[
  \Big(\prod\limits_{i=1}^d (1+|2\pi \xi_i|^{2r_i})\Big)^{1/2} \mathcal{F}\varphi(\boldsymbol{\xi})\Big]\Big\|_2\\
  &= \Big\|\Big(\prod\limits_{i=1}^d (1+|2\pi \xi_i|^{2r_i})\Big)^{1/2}\mathcal{F}\varphi(\boldsymbol{\xi})\Big\|_2
  = \|v_{\bsr}(\boldsymbol{\xi})\mathcal{F}\varphi\|_2\,,
\end{split}
\end{equation}
where we define
\begin{equation}
  v_{\mathbf{r}}(\bsx):= \Big(\prod\limits_{i=1}^d (1+|2\pi x_i|^{2r_i})\Big)^{1/2}\,.
\end{equation}
Let now $\Omega$ be a bounded domain in $\R^d$. We denote with $C^{\infty}_0(\Omega)$ the space of all infinitely many times
differentiable (real-valued) functions $\varphi:\R^d \to \R$ with $\supp \varphi \subset \Omega$.
Finally, we define the space
\begin{equation}\label{f1}
  \mathring{H}^{\mathbf{r}}_{\text{mix}}(\overline{\Omega}) := \overline{C_0^{\infty}(\Omega)}^{\|\cdot\|_{H^{\mathbf{r}}_{\text{mix}}}}
\end{equation}
by completion with respect to the norm $\|\cdot\|_{H^{\mathbf{r}}_{\text{mix}}}$\,. As a consequence we get that 
$\mathring{H}^{\mathbf{r}}_{\text{mix}}(\overline{\Omega})$ is a Hilbert space which  consists
of $r_i-1$ times continuously differentiable functions (mixed in each component) on $\R^d$ 
which vanish on $\R^d\setminus \Omega$\,.

We will now consider a more specific situation. Let $\Omega = (0,1)^d$. Then it holds 
\begin{equation}\label{tensor}
    \mathring{H}^{\mathbf{r}}_{\text{mix}}:=\mathring{H}^{\mathbf{r}}_{\text{mix}}([0,1]^d) = \mathring{H}^{r_1}([0,1]) \otimes \cdots \otimes \mathring{H}^{r_d}([0,1])
\end{equation}
in the sense of tensor products of Hilbert spaces, where $\mathring{H}^{r} = \mathring{H}^{r_i}([0,1])$ is the univariate version of the above
defined spaces. Functions in this class satisfy a left and a right boundary condition, namely 
$f^{(j)}(0) = f^{(j)}(1) = 0$ for $j=0,...,r-1$. 

The first assertion in the following lemma is a direct consequence of Taylor's theorem and the homogeneous 
boundary condition of the function and all its derivatives. The second one follows from (i) together 
with H\"older's inequality. 
\begin{lemma}\label{lem51} Let $\bsr \in \N^d$. 
{\em (i)} Every function $\varphi \in C_0^{\infty}((0,1)^d)$ admits the following
representation
$$
    \varphi(x_1,...,x_d) = \int_0^1 \cdots \int_0^1 \varphi^{(\bsr)}(t_1,...,t_d)\prod\limits_{i=1}^d
   \frac{(x_i-t_i)^{r_i-1}_+}{(r_i-1)!}\,dt_1...dt_d\,.
$$
{\em (ii)} Let $e \subset [d]$. Then 
$$
  \Big\|\Big(\prod\limits_{i\in e} \frac{\partial^{r_i}}{\partial x_i^{r_i}}\Big)\varphi\Big\|^2_2 
  \leq \|\varphi^{(\bsr)}\|^2_2\prod\limits_{i\in e} \frac{1}{[(r_i-1)!]^2(2r_i-1)2r_i} 
$$
and therefore
$$
  \|\varphi\|^2_{H^{\bsr}_{\text{mix}}} \leq \|\varphi^{(\bsr)}\|^2_2\sum\limits_{e \subset [d]} 
  \prod\limits_{i\in e} \frac{1}{[(r_i-1)!]^2(2r_i-1)2r_i} \,.
$$
\end{lemma}

\begin{remark} {\em (a)} Note, that the assertions in Lemma \ref{lem51} hold true for 
any function $\varphi \in S(\R^d)$ with $\supp \varphi \subset \R_+^d$\,. We only need zero 
boundary values at $0$.\\

{\em (b)} The previous lemma shows that the semi-norm $\|\cdot\|_{\mr{H}^{\bsr}_{\text{mix}}}$
induced by the bilinear form 
\begin{equation}\label{f2}
   \langle \varphi, \psi \rangle_{\mr{H}^{\bsr}_{\text{mix}}}:= \int_{[0,1]^d} 
   \varphi^{(\bsr)}(\bsx)\psi^{(\bsr)}(\bsx)\,d\bsx
\end{equation}
is actually a norm on $C^{\infty}_0((0,1)^d)$ since the bilinear form is positive definite as a consequence of (ii). 
Hence, we could have also used this semi-norm for 
the completion in \eqref{f1}. As it turns out Lemma \ref{lem51} and \eqref{f2} are actually
the key to derive the reproducing kernel for the space $\mathring{H}^{\bsr}_{\text{mix}}$.\\

{\em (c)} We have an explicit upper 
bound for the norm equivalence constant in (ii). Suppose that we have a constant 
smoothness vector $\bsr = (r,...,r)$ with $r \in \N$\,. Then it holds 
\begin{equation}\label{f12}
\begin{split}
\sum\limits_{e \subset [d]} 
  \prod\limits_{i\in e} \frac{1}{[(r_i-1)!]^2(2r_i-1)2r_i} &= \sum\limits_{i=0}^d 
  \binom{d}{i} \Big(\frac{1}{[(r-1)!]^2(2r-1)2r}\Big)^i\\
  &= \Big(1+\frac{1}{[(r-1)!]^2(2r-1)2r}\Big)^d\,.
\end{split}  
\end{equation}
Hence, if $r=1$ the constant is bounded by $(3/2)^d$, in case $r=2$ we have $(13/12)^d$ and in case $r=3$
already $(61/60)^d$\,.

\end{remark}

\subsection{The reproducing kernel of $\mathring{H}^{\mathbf{r}}_\mix$} \label{sec:repro_kernel}
In the sequel we will identify the space $\mathring{H}^{\mathbf{r}}_{\text{mix}}$ as 
a reproducing kernel Hilbert space. We are 
looking for a kernel function $\mr{K}^{\bsr}_d(\bsx,\bsy)$ such that for every 
$f\in \mathring{H}^{\mathbf{r}}_{\text{mix}}$
$$
    \langle f(\cdot), \mr{K}^{\bsr}_d(\bsx,\cdot)\rangle_{\mr{H}^{\bsr}_{\text{mix}}}  = f(\bsx)\quad,\quad \bsx \in [0,1]^d\,.
$$
To this end, we may derive the reproducing kernels of the univariate spaces $\mathring{H}^{r_i}$. The reproducing kernel of the tensor product space \eqref{tensor} is then given by the point-wise product of the univariate kernels
\begin{equation}
	\mathring{K}^{\bsr}_{d}(\bsx, \bsy) = \prod_{\ell=1}^d \mathring{K}^r_{1}(x_\ell, y_\ell) .
\end{equation}
Therefore, the problem of computing $\mathring{K}_{d}^\bsr(\bsx, \bsy)$ is reduced to the construction of $\mathring{K}^r_{1}: [0,1] \times [0,1] \to \R$. 

Let us first recall a general fact for Hilbert spaces and orthogonal sums.
To this end, let $U:=\mathrm{span}\{u_0,\ldots,u_{r-1}\} \subset \cH$ be an $r$-dimensional subspace of a 
Hilbert space $\cH$. Using Gram-Schmidt orthogonalization, the orthogonal projection $P_U: \cH \to U$ is given by
\begin{equation} \label{eqn_orth_projection}
	P_U(f)(x) = \sum_{j=0}^{r-1} \left( \sum_{k=0}^{r-1} G^{-1}_{j,k} \cdot \langle f,u_k\rangle_{\cH} \right) u_j ,
\end{equation}
where the  Gramian matrix $\mathbf{G} = \left( \langle u_j, u_k \rangle_{\cH} \right)_{j,k=0}^{r-1} \in \R^{r \times r}$. 
Moreover, the projection onto the orthogonal complement $U^\perp = \cH \ominus U$ is $P_{U^\perp}f = (\mathrm{Id} - P_U)f$.

The next Lemma provides the necessary utilities to compute the reproducing kernel of closed subspaces 
that are defined via homogeneous boundary conditions. 
		
\begin{lemma}\label{modify}
Let $\cH_K$ be a RKHS with kernel $K: [0,1] \times [0,1] \to \R$.  
Assuming that $K(x,\cdot)$ is $r$ times weakly differentiable, 
let $u_{j} := K^{(0,j)}(\cdot, 1) := 
\frac{\partial^j}{\partial y^j} K(\cdot, y)_{|y=1}$ for $j=0,\ldots,r-1$ and $U = 
\mathrm{span}\{u_0,\ldots, u_{r-1}\}$. Then it holds that

	\begin{enumerate}
		\item[(i)] For $j=0,\ldots,r-1$, the Riesz representer of the functional $f \mapsto f^{(j)}(1)$ in 
		$\cH_K$ is given by $u_{j}$, i.e.
		\[
			\langle f, u_{j} \rangle_{\cH_K} = f^{(j)}(1) \quad \text{ for all } f \in \cH_K .
		\]
		\item[(ii)] The reproducing kernel $K_{U^\perp}$ of $U^\perp \subset \cH_K$, i.e. the orthogonal complement of $U$ in $\cH_K$, 
		is given by
		\begin{equation}\label{f6}
			K_{U^\perp}(x,y)  = P_{U^\perp} K(\cdot, y)(x) = K(x,y) - \sum_{j=0}^{r-1} 
			\sum_{k=0}^{r-1} G^{-1}_{j,k} u_j(x) u_k(y) .
		\end{equation}
		\item[(iii)] It holds that
		$$
			U^\perp = \{f \in \cH_K: f^{(j)}(1)=0, j=0,\ldots,r-1 \} .
		$$
	\end{enumerate}
	\begin{proof}
		(i) is \cite[Lem. 10]{berlinet} for the linear functional $f \mapsto f^{(j)}(1)$ and 
		(ii) follows by applying \cite[Thm. 11]{berlinet} to \eqref{eqn_orth_projection}. 
		Finally, regarding (iii) we note that it holds for all $f \in U^\perp$ that 
         	$$\langle f, u_j \rangle_{\cH_K} = \langle f, K^{(0,j)}(\cdot, 1) \rangle_{\cH_K} = f^{(j)}(1) = 0 .$$
	\end{proof}
\end{lemma}
We want to apply this machinery to $\mr{H}^r$ with $r\in \N$. The observation in Lemma \ref{lem51} together 
with \eqref{f2} gives rise to use the 
approach of Wahba \cite[1.2]{Wah90} as a starting point. Let us define the 
kernel function
\begin{equation}\label{ker}
  K^r_1(x,y) := \int_0^1 \frac{(x-t)^{r-1}_+}{(r-1)!}\cdot \frac{(y-t)^{r-1}_+}{(r-1)!}\,dt\quad,\quad x,y\in [0,1]\,.
\end{equation}
Then it is immediately clear from Lemma \ref{lem51},(i) (and a straight-forward density argument) that 
$$
    f(x) = \langle f(\cdot), K_1^r(x,\cdot) \rangle_{\mr{H}^{r}}\quad, \quad x\in [0,1]\,.
$$
Indeed, recall that the inner product $\langle\cdot,\cdot \rangle_{\mr{H}^r}$ stems from \eqref{f2} and that 
$$
   (K^r_1)^{(0,r)}(x,y) = \frac{(x-y)^{r-1}_+}{(r-1)!}\,.
$$
It is possible to give an explicit formula for \eqref{ker} by using that 
\begin{equation}\label{f4}
  K^r_1(x,y) := \int_0^{\min\{x,y\}} \frac{(x-t)^{r-1}_+}{(r-1)!}\cdot \frac{(y-t)^{r-1}_+}{(r-1)!}\,dt\,.
\end{equation}
Interpreting this as a Taylor remainder term we find 
\begin{equation}\label{f5}
K^r_1(x,y) = \frac{(-1)^r}{(2r-1)!}\Big[\sum\limits_{k=r}^{2r-1}\binom{2r-1}{k}(-\min\{x,y\})^k\max\{x,y\}^{2r-1-k}\Big]\,.
\end{equation}
However, $\mr{H}^r$ is a only a closed subspace of $\mathcal{H}_{K_1^r}$ since the functions 
$f \in \mathcal{H}_{K_1^r}$ may lack the right boundary condition which is $f^{(j)}(1) = 0$ 
if $j=0,...,r-1$, whereas the left boundary condition $f^{(j)}(0) = 0$ if $j=0,...,r-1$ is for free due to the construction.
Let us now apply the construction from Lemma \ref{modify} to $K^r_1$ to construct a 
reproducing kernel $\mr{K}^r_1$ for the 
closed subspace $\mr{H}^r$\,.

\begin{figure}[t]
 \includegraphics[width=0.49\linewidth]{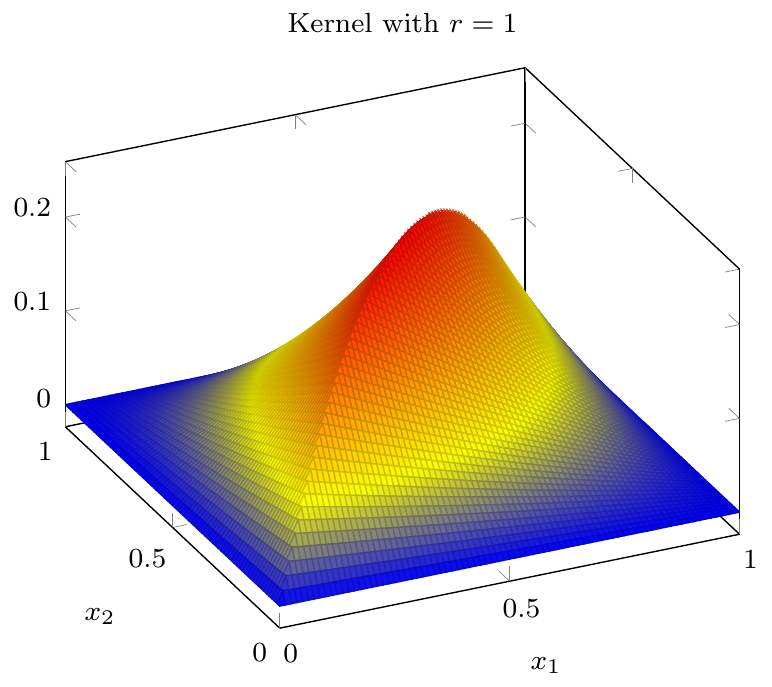}
 \,
 \includegraphics[width=0.49\linewidth]{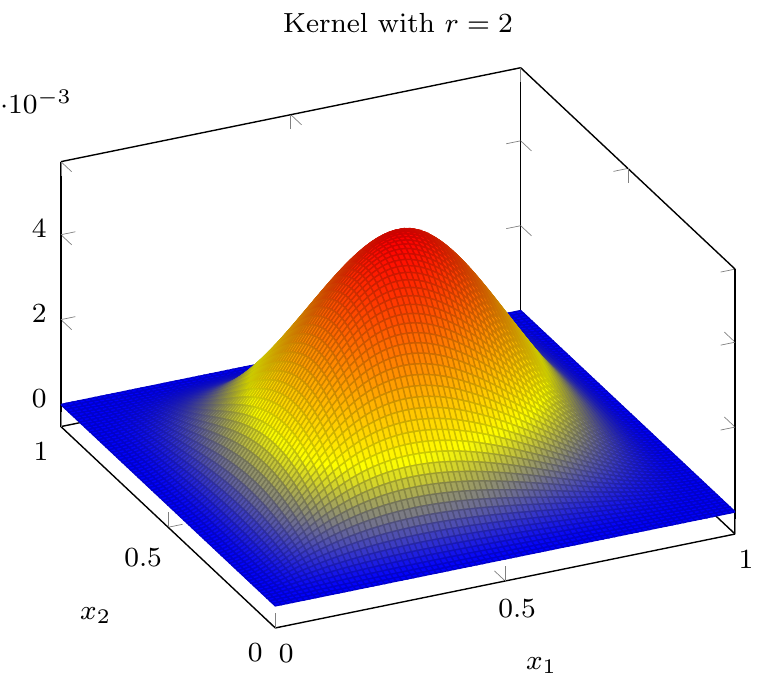}
 \caption{Plots of the kernel $\mathring{K}^r_1: [0,1] \times [0,1] \to \R$ with smoothness $r=1$ (left) and smoothness $r=2$ (right).} \label{fig_kernel_plot}
\end{figure}
First we compute the functions $u_j(\cdot) = (K_1^r)^{(0,j)}(\cdot,1)$ for $j = 0,...,r-1$
explicitly. Using again the formula \eqref{ker} we find
\begin{equation}\label{f7_1}
  \begin{split}
    u_j(x) &= \Big(\frac{d}{dy}\Big)^j\int_0^{\min\{x,y\}} \frac{(x-t)^{r-1}_+}{(r-1)!}\cdot \frac{(y-t)^{r-1}_+}{(r-1)!}\,dt{\bigg \vert}_{y=1}\\
    &=\int_0^x \frac{(x-t)^{r-1}_+}{(r-1)!}\cdot \frac{(1-t)^{r-1-j}}{(r-1-j)!}\,dt\,,
  \end{split}
\end{equation}
where we used the well-known formula for the differentiation of integrals. 
Similar as above in \eqref{f4} we interpret this as a Taylor's remainder term for a specific polynomial. 
It is not hard to verify that this polynomial 
is given by
\begin{equation}\label{f7_b}
  u_j(x) = \frac{(-1)^r}{(2r-1-j)!} \Big[\sum\limits_{k=r}^{2r-1-j} 
  \binom{2r-1-j}{k}(-x)^k\Big]\,\quad,\quad j=0,...,r-1\,.
\end{equation}
Looking at the functions $u_j$, $j=0,...,r-1$, we see immediately that 
$\{x^r,...,x^{2r-1}\}$ is a basis of their span. Hence we may use the system 
$\tilde{u}_j(x) = x^{j+r}/(j+r)!$ in \eqref{f6}\,. This gives the following representation for the 
kernel $\mathring{K}^r_{1}(x,y)$, namely
\begin{equation}\label{f7_2}
    	\mathring{K}^r_{1}(x,y) = K_1^r(x,y)  - \sum_{j=0}^{r-1} \sum_{k=0}^{r-1} \frac{G^{-1}_{j,k}}{(j+r)!(k+r)!}
    	x^{j+r} y^{k+r}\,,
\end{equation}
where $K_1^r(x,y)$ is given by \eqref{f5} and 
$$\mathbf{G} = \Big(\frac{1}{j! k!(j+k+1)}\Big)_{\substack{j=0,...,r-1\\k=0,...,r-1}}\,.$$

Let us give two examples. Putting $r=1$ in \eqref{f7_2} we have
$$
    \mr{K}^1_1(x,y) = \min\{x,y\} - xy,\quad,\quad x,y \in [0,1]\,.
$$
Furthermore, in case $r=2$ we obtain 
$$
  \mr{K}^2_1(x,y) = K^2_1(x,y) - x^2y^2+x^2y^3/2+x^3y^2/2-x^3y^3/3\,,
$$
where 
$$
    K^2_1(x,y) = \frac{1}{2}\min\{x,y\}^2\max\{x,y\}-\frac{1}{6}\min\{x,y\}^3,
$$
For $r=1,2,3$ we obtain the associated Gramian matrices
$$
    (\mathbf{G}^1)^{-1} = \left(\begin{matrix}
                             1
                          \end{matrix}\right)\quad,\quad
    (\mathbf{G}^2)^{-1} = \left(\begin{matrix}
                             4 & -6\\
                             -6 & 12
                          \end{matrix}\right)\quad,\quad
    (\mathbf{G}^3)^{-1} = \left(\begin{matrix}
                             9& -36 & 60\\
                             -36 & 192 & -360\\
                             60 & -360 & 720
                          \end{matrix}\right)\,.
$$

In the case $d=1$ the kernels for $r=1$ and $r=2$ are depicted in Figure \ref{fig_kernel_plot}. The smoothness can be observed along the 
diagonal $x = y$, where the kernel for $r=1$ exhibits a kink.

Regarding the multivariate kernel, we have arrived at the following result.

\begin{theorem}\label{repr_kernel}
 Given a smoothness vector $\mathbf{r} = (r_1, r_2,\ldots,r_d) \in \N^d$, the reproducing kernel of the tensor 
 product space $\mathring{H}^{\mathbf{r}}_\mix = \mathring{H}^{r_1} \otimes \cdots \otimes \mathring{H}^{r_d}$ is given by
 \begin{align}
  \mathring{K}_{d}^{\mathbf{r}}(\bsx, \bsy) & = \prod_{\ell=1}^d \mathring{K}_{1}^{r_\ell}(x_\ell, y_\ell) \\
      & = \prod_{\ell=1}^d  \Big(K_{1}^{r_\ell}(x_\ell, y_\ell)  - 
      \sum_{j=0}^{r_\ell-1} \sum_{k=0}^{r_\ell-1} (\mathbf{G}^{r_\ell})^{-1}_{j,k} (K_{1}^{r_\ell})^{(0,j)}(x_\ell,1) 
      \, (K_{1}^{r_\ell})^{(0,k)}(y_\ell,1)  \Big),
 \end{align}
 where $u_j (x_\ell) = (K_{1}^{r_\ell})^{(0,j)}(x_\ell,1)$ are given in \eqref{f7_b} 
 and $(\mathbf{G}^{r_\ell})^{-1}$ are given in \eqref{f2}.
\end{theorem}

The explicit expression for the reproducing kernel of $\mathring{H}^{\mathbf{r}}_\mix$ allows to 
compute the norms of arbitrary bounded linear functionals $L \in (\mathring{H}^{\mathbf{r}}_\mix)^\star$, since it holds
\begin{equation} \label{eqn_functional_norm}
 \|L\|_{(\mathring{H}^{\mathbf{r}}_\mix)^\star} =  \sup_{\|f\|_{ \mathring{H}^{\mathbf{r}}_\mix } 
 \leq 1} |L(f)|  = \sqrt{ L^{(\bsx)} L^{(\bsy)} \mathring{K}_{d}^{\mathbf{r}}(\bsx, \bsy) }.
\end{equation}
The right-hand side involves the application of the functional 
$L$ to both components of the kernel. We will use this in Section 
\ref{sec:numerics} for the simulation of worst-case integration errors which can be 
rewritten as norms of certain functionals \eqref{eqn_wce_formula} involving the integration functional 
$L(f) = I_d(f) = \int_{[0,1]^d} f(\bsx) \, \rd \bsx$. In the sequel we will compute 
the norm and its Riesz representer. We have  
\begin{equation} \label{eqn_initial_error}
   \|I_d\|^2_{(\mathring{H}^{\mathbf{r}}_\mix)^\star} =  \sup_{\|f\|_{ \mathring{H}^{\mathbf{r}}_\mix } \leq 1} |I_d(f)|^2   = 
   \prod_{\ell=1}^d \left( \int_0^1 \int_0^1 \mathring{K}_{1}^{r_\ell}(x_\ell, y_\ell) \, \rd x_\ell \rd y_\ell \right)
\end{equation}
where
 \begin{align*}
  \int_0^1 \int_0^1 \mathring{K}_{1}^{r}(x, y) \, \rd x \rd y  = & \int_0^1 \int_0^1 K_1^r(x,y)  -  \sum_{j=0}^{r-1} \sum_{k=0}^{r-1} \frac{G^{-1}_{j,k}}{(j+r)!(k+r)!} 
    	x^{j+r} y^{k+r} \, \rd x \rd y \\
    	= & \int_0^1 \int_0^1 K_1^r(x,y)  \, \rd x \rd y - \sum_{j=0}^{r-1} \sum_{k=0}^{r-1} \frac{G^{-1}_{j,k}}{(j+r + 1)!(k+r + 1)!} \\
    	= & \frac{1}{ (r!)^2 (2r+1)} - \sum_{j=0}^{r-1} \sum_{k=0}^{r-1} \frac{G^{-1}_{j,k}}{(j+r + 1)!(k+r + 1)!}\,.
 \end{align*}
 The last identity follows from the representation \eqref{f4} and
 \begin{align*}
  \int_0^1 \int_0^1 K_1^r(x,y)  \, \rd x \rd y & = \int_0^1 \left( \int_0^1 \int_0^1  \frac{(x-t)^{r-1}_+}{(r-1)!}  \frac{(y-t)^{r-1}_+}{(r-1)!} \,  \rd x \rd y \right) \rd t \\
		  & = \int_0^1 \left(  \int_t^1  \frac{(x-t)^{r-1}}{(r-1)!}  \int_t^1 \frac{(y-t)^{r-1}}{(r-1)!} \,  \rd x \rd y \right) \rd t \\
		  & = \int_0^1 \left( \frac{(1-t)^r}{r!}  \frac{(1-t)^r}{r!}  \right) \rd t = \int_0^1 \frac{(1-t)^{2r}}{(r!)^2}  \ \rd t \\
		  & = \frac{1}{(2r)! (2r+1) } .
 \end{align*}
 For the Riesz representer of $f \mapsto \int_{[0,1]^d} f(\bsx) \rd\bsx = \langle f, R_{I_d} \rangle_{\mr{H}^{\bsr}_{\text{mix}}}$ it holds 
 $$
   R_{I_d}(\bsy) = \int_{[0,1]^d} \mathring{K}_{d}^{\mathbf{r}}(\bsx, \bsy)\,\rd \bsx 
   = \prod_{\ell=1}^d \left( \int_0^1 \mathring{K}_{1}^{r_\ell}(x_\ell, y_\ell) \, \rd x_\ell  \right)\quad,\quad 
   \bsy = (y_1,...,y_d)\,.
 $$
 Clearly, we have
 $$
    \int_0^1 \mathring{K}_{1}^{r_\ell}(x, y) \, \rd x = \int_0^1 K_{1}^{r_\ell}(x, y)  - 
    \sum_{j=0}^{r-1} \sum_{k=0}^{r-1} \frac{G^{-1}_{j,k}}{(j+r)!(k+r)!} 
    	x^{j+r} y^{k+r} \, \rd x\,. 
 $$ 
 A similar computation as above together with the identity
 $$
    \int_0^1 \frac{(y-t)^{r-1}_+}{(r-1)!}\frac{(1-t)^r}{r!}\rd t = \frac{(-1)^r}{(2r)!}\sum\limits_{k=r}^{2r}
    \binom{2r}{k}(-y)^k 
 $$
 (see the computation after \eqref{f7_1}) leads to the following explicit formula 
 \begin{equation} \label{eqn_riesz_representer}
  \int_0^1 \mathring{K}_{1}^{r_\ell}(x, y) \, \rd x =\frac{(-1)^r}{(2r)!}\sum\limits_{k=r}^{2r} \binom{2r}{k}(-y)^k - \sum_{j=0}^{r-1} \sum_{k=0}^{r-1} \frac{G^{-1}_{j,k}}{(j+r+1)!(k+r)!}y^{k+r}\,. 
 \end{equation}

\section{Worst-case error estimates with respect to $\mathring{H}^{\mathbf{r}}_{\text{mix}}$}\label{sec:TheorBounds}
In this section, we are interested in the behavior of the worst-case error 
\begin{equation}\label{eq:error}
e(\scalingFactor,d,\bsr) \,:=\, \sup_{\|f\|_{\mr{H}^{\bsr}_{\text{mix}} \leq 1}}\, \abs{Q_\scalingFactor^d(f)-I_d(f)}.
\end{equation}
of Frolov's cubature rule $Q_\scalingFactor^d$ with respect to the unit ball
in the norm $\|\cdot\|_{\mr{H}^\bsr_{\text{mix}}}$, see \eqref{f2}. 
Recall that 
\[
Q_\scalingFactor^d(f) \,=\, \frac{1}{\scalingFactor}\sum_{\bsk\in\Z^d}\, f(\bm{A}_\scalingFactor \bsk),
\]
where $\bm{A}_\scalingFactor=\scalingFactor^{-1/d}\bm{A}$ and $\bm{A}=\bigl(\det(\bm{V})\bigr)^{-1/d}\bm{V}$
with $\bm{V}$ from Theorem~\ref{construction}. Let further $\bm{B}_\scalingFactor = (\bm{A}_\scalingFactor)^{-\top}$

The main tool for analyzing \eqref{eq:error} is 
Poisson's summation formula. Let $\varphi \in S(\R^d)$ be a multivariate 
Schwartz-function. With $\mathcal{F}\varphi$ we denote the Fourier transform
\begin{equation}\label{FT}
  \mathcal{F}\varphi(\xi) = \int_{-\infty}^{\infty} \varphi(x)\exp(-2\pi i\bsx\cdot\mathbf{\xi})\,d\bsx\quad,\quad \mathbf{\xi}\in \R^d\,.
\end{equation}
Then it holds 
$$
  \sum\limits_{\bsm \in \Z^d} \varphi(x+\bsm) = \sum\limits_{\bsk \in \Z^d} \mathcal{F}\varphi(\bsk)\exp(2\pi i\bsk\cdot \bsx)
$$
with absolute convergence on both sides. The following consequence is of particular importance. Let 
$\bm{A}:\R^d \to \R^d$ be a regular matrix with 
$\det \bm{A} \neq 0$. Let further $\bm{B}=\bm{A}^{-\top}$\,. Then 
we have 
\begin{equation}\label{poiss1}
  \det \bm{A} \sum\limits_{\bsm \in \Z^d} \varphi(\bm{A}(\bsx+\bsm)) = \sum\limits_{\bsk \in \Z^d} \mathcal{F}
  \varphi(\bm{B}\bsk)\exp(2\pi i\bsk\cdot \bsx)
\end{equation}
Let us finally mention the following special case by putting $\bsx = 0$
\begin{equation}\label{poiss2}
  \det \bm{A} \sum\limits_{\bsm \in \Z^d} \varphi(\bm{A}\bsm) = 
  \sum\limits_{\bsk \in \Z^d} \mathcal{F}\varphi(\bm{B}\bsk)
\end{equation}
A more general variant (with respect to the regularity of the participating functions) can be found in \cite[Thm.\ 3.1, Cor.\ 3.2]{UlUl16}
 
In this section we show the by now well-known upper bounds
on the worst-case error of Frolov's cubature formula for the 
Sobolev spaces $\mr{H}^{\mathbf{r}}_{\text{mix}}$.
We give relatively short proofs here with special emphasis on the 
constants. In particular, we will see how the invariants of the used lattice
will affect the error estimates. 

We will see that only two invariants will play a role in the upper bounds, 
which we want to discuss shortly. 
For this note that the lattices under consideration are generated by a 
multiple of a 
Vandermonde matrix  $\bm{V}$, which is defined via a generating polynomial $P$ 
as in Theorem~\ref{construction}. 
The first invariant is the determinant,
or in other words the discriminant of the generating polynomial
\[
 D_P \,=\, \det(\bm{V}).
\]
For example, we know from Theorem~\ref{construction} that 
$\Nm(\bm{V}^{-\top})
=1/D_P^{2}$.

The second invariant is 
\begin{equation}\label{eq:b}
B_P \,:=\, \min_{\bm{U}} \|\bm{V}\bm{U}\|_\infty,
\end{equation}
where the minimum is over all $\bm{U}\in\SLdZ$.
This constant is an upper bound for the diameter (in $\ell_\infty$) of 
the ``smallest'' fundamental cell of the lattice.
To see this, note that every \emph{fundamental cell}, i.e.~a parallelepiped 
with corners on the lattice with no lattice point in the interior, 
is of the form $T([0,1]^d)$, where $T\in\R^{d\times d}$ is a generating matrix
for the lattice. Moreover, it is well-known that every generating matrix of the 
lattice that is generated by $\bm{V}$ is of the form $\bm{V}\bm{U}$ for some unimodular,
integer-valued matrix $\bm{U}$. We will see that both, $D_P$ and $B_P$, should be small to obtain a small upper 
bound on the errors. This justifies the choice of the generating polynomials in the 
previous section. Here is the main result of this section.

\begin{theorem}\label{thm:theor-error}
Let $\bsr=(r_1,\dots,r_d) \in \N^d$ and $\eta:=\#\{j\colon r_j=r\}$. Then we have for any $f\in\mr{H}^\bsr_{\mix}$ 
\begin{equation}\label{f13}
\hspace{-.5cm}
\begin{split}
&\abs{Q_\scalingFactor^d(f)-I_d(f)}\\
&\le\; C(d,\eta,\bsr) \cdot \max\left\{D_P, \frac{(2 B_P)^d}{n}\right\}^{1/2}
\left(\frac{D_P}{n}\right)^{r}\,\Bigl(2+\log\left(n/D_P\right)\Bigr)^{(\eta-1)/2}
\,\norm{f}_{\mr{H}^{\bsr}_{\text{mix}}},
\end{split}
\end{equation}
where 
\begin{equation}\nonumber
\begin{split}
  &C(d,\eta,\bsr)\\
  &:=\, 2^{d+1}\, \left(1-2^{-2(r'-r)}\right)^{-(d-\eta)/2}\,
\left(1-2^{(1-2r)}\right)^{-\eta/2}\left(\sum\limits_{e \subset [d]} 
  \prod\limits_{i\in e} \frac{1}{[(r_i-1)!]^2(2r_i-1)2r_i}\right)^{1/2}
\end{split}
\end{equation}
with $r':=\min_j\{r_j\colon r_j\neq r\}$.
\end{theorem}

Let us prove the following estimate first. 

\begin{proposition} Let $\varphi \in C_0^{\infty}((0,1)^d)$. Then 
\begin{equation}\label{f11}
   \abs{Q_\scalingFactor^d(\varphi)-I_d(\varphi)} \leq 
   \sqrt{\frac{M(\bm{A}_\scalingFactor)}{\det \bm{B}_\scalingFactor}}\Big(\sum\limits_{\bsk \in \Z^d\setminus \{\mathbf{0}\}} |v_{\bsr}(
   \bm{B}_\scalingFactor\bsk)|^{-2}\Big)^{1/2} \|\varphi\|_{H^{\bsr}_\text{mix}}\,,
\end{equation}
where 
\begin{equation}\label{f9}
  M(\bm{A}_\scalingFactor) \,:=\, \min_{\bm{U}\in\SLdZ}\#\Bigl\{\bsm\in\Z^d\colon \bm{UA}_\scalingFactor\bigl(\bsm+(0,1)^d\bigr)\cap[0,1]^d\neq\emptyset\Bigr\}\,,
\end{equation}
is the minimal number of fundamental cells of the integration lattice necessary 
to cover the unit cube. 
\end{proposition}

\begin{proof} The above special case of Poisson's summation formula \eqref{poiss2} gives 
\begin{equation}\label{f7}
 \begin{split}
  &\abs{Q_\scalingFactor^d(\varphi)-I_d(\varphi)} = \Big|\sum\limits_{\bsk\in \Z^d\setminus \{\mathbf{0}\}}
  \mathcal{F}\varphi(\bm{B}_\scalingFactor \bsk)\Big|\\
  &~~\leq \Big(\sum\limits_{\bsk\in \Z^d\setminus \{\mathbf{0}\}}v_{\mathbf{r}}(\bm{B}_\scalingFactor\bsk)^{-2}\Big)^{1/2}
  \Big(\sum\limits_{\bsk \in \Z^d}  |v_{\mathbf{r}}(\bm{B}_\scalingFactor\bsk)
  \mathcal{F}\varphi(\bm{B}_\scalingFactor\bsk)|^2\Big)^{1/2}\\
 \end{split}
\end{equation}
By the definition of $v_{\bsr}$ we may rewrite
$$
  |v_{\mathbf{r}}(\bm{B}_\scalingFactor\bsk)
  \mathcal{F}\varphi(\bm{B}_\scalingFactor\bsk)|^2= \sum\limits_{e \subset [d]}\Big|\mathcal{F}\Big[\Big(\prod\limits_{i\in e} 
  \frac{\partial^{r_i}}{\partial x_i^{r_i}}\Big)\varphi\Big](\bm{B}_\scalingFactor\bsk)\Big|^2
$$
Using this for the second factor in \eqref{f7} we find 
\begin{equation}\nonumber
  \begin{split}
    &\sum\limits_{\bsk \in \Z^d}  |v_{\mathbf{r}}(\bm{B}_\scalingFactor\bsk)
  \mathcal{F}\varphi(\bm{B}_\scalingFactor\bsk)|^2\\
  &~~=\sum\limits_{e \subset [d]}\int\limits_{[0,1]^d}
  \Big|\sum\limits_{\bsk \in \Z^d}\mathcal{F}\Big[\Big(\prod\limits_{i\in e} 
  \frac{\partial^{r_i}}{\partial x_i^{r_i}}\Big)\varphi\Big](\bm{B}_\scalingFactor\bsk)\exp(2\pi i\bsk\cdot \bsx)\Big|^2\,d\bsx\,.
  \end{split}
\end{equation}
Now we apply Poisson's summation formula in the form \eqref{poiss1} to the integrand and find
\begin{equation}
  \begin{split}
    &\sum\limits_{\bsk \in \Z^d}  |v_{\mathbf{r}}(\bm{B}_\scalingFactor\bsk)
  \mathcal{F}\varphi(\bm{B}_\scalingFactor\bsk)|^2\\
  &~~=(\det \bm{A}_\scalingFactor)^2\sum\limits_{e \subset [d]}\int\limits_{[0,1]^d}
  \Big|\sum\limits_{\bsm \in \Z^d}\Big[\Big(\prod\limits_{i\in e} 
  \frac{\partial^{r_i}}{\partial x_i^{r_i}}\Big)\varphi\Big](\bm{A}_\scalingFactor(\bsx+\bsm))\Big|^2\,d\bsx\\
  &~~\leq(\det \bm{A}_\scalingFactor)^2M(\bm{A}_\scalingFactor)\sum\limits_{e \subset [d]}\sum\limits_{\bsm \in \Z^d}
  \int\limits_{[0,1]^d}\Big|\Big[\Big(\prod\limits_{i\in e} 
  \frac{\partial^{r_i}}{\partial x_i^{r_i}}\Big)\varphi\Big](\bm{A}_\scalingFactor(\bsx+\bsm))\Big|^2\,d\bsx\\
  &~~=(\det \bm{A}_\scalingFactor)M(\bm{A}_\scalingFactor)\sum\limits_{e \subset [d]}
  \int_{\R^d}\Big|\Big[\Big(\prod\limits_{i\in e} 
  \frac{\partial^{r_i}}{\partial x_i^{r_i}}\Big)\varphi\Big](\bsy)\Big|^2\,d\bsy\\
  &= \frac{M(\bm{A}_\scalingFactor)}{\det \bm{B}_\scalingFactor}\|\varphi\|^2_{H^{\bsr}_{\text{mix}}}\,,
  \end{split}
\end{equation}
where we used H\"older's inequality and the fact that $\varphi$ and all its partial derivatives have 
compact support in $(0,1)^d$ together with \eqref{f9}. 
\end{proof}
\begin{remark} Let us comment on the number $M(\bm{A}_\scalingFactor)$. Clearly, all the 
fundamental cells are contained in 
$[-L(\scalingFactor,P),1+L(\scalingFactor,P)]^d$ with $L(\scalingFactor,P):=(D_P \scalingFactor)^{-1/d} B_P$ and
$B_P$ from~\eqref{eq:b}. Here, we used that $\bm{A}_\scalingFactor=(D_P \scalingFactor)^{-1/d}\bm{V}$.
Therefore, $M(\bm{A}_\scalingFactor)$ is bounded 
by the number of lattice points $\bm{A}_\scalingFactor(\Z^d)$ in this set. 
This number can be controlled by \eqref{f10} below, which will be also
of some importance later. For a proof see e.g.~\cite[Lem.\ 5]{MU17}. 
In fact, for every axis-parallel box $\Omega\subset\R^d$ and every $T\in\R^{d\times d}$ 
we have 
\begin{equation}\label{f10}
\#\left(\bm{T}(\Z^d)\cap\Omega\right) \,\le\, 
\frac{\vol_d(\Omega)}{\Nm(\bm{T})} +1.
\end{equation}
With all the definitions from above and $\Nm(\bm{V})=1$, we obtain that 
\begin{equation}\label{eq:M}
M(\bm{A}_\scalingFactor) \,\le\, n\, D_P\,\left(1+\frac{2 B_P}{(D_P n)^{1/d}}\right)^d + 1 
\,\le\, n\,2^d\, \max\left\{D_P,\, \frac{(2 B_P)^d}{n}\right\}.
\end{equation}
We see that the bound of the second factor of the above error bound depends 
asymptotically only on $\sqrt{D_P}$ (and the norm of $f$). 
However, for preasymptotic bounds also the term $B_P^{d/2}/\sqrt{n}$ 
plays an important role. 
\end{remark}

\begin{proof}
To finish the proof of Theorem \ref{thm:theor-error} it remains to estimate the middle factor in \eqref{f11}. 
In fact, the statement \eqref{f13} then follows by a straight-forward density argument recalling \eqref{f1}. 

If $\bsr=(r,\dots,r)$ with $r\in\N_0$ is a constant smoothness vector, the following proof can be found 
in several articles, see e.g.~\cite{tem93} or \cite[p.~580]{MU16}. Note, that it also works for fractional 
$r>1/2$, which is essentially shown in \cite{UlUl16}. 
Although the optimal order of convergence is known also in the non-constant case, 
we were not able to find a proof with explicit constants. 
Therefore, we give it here.
We assume without restriction that $r_1=\dots= r_\eta<r_{\eta+1}\le\dots\le r_d$ 
for some $\eta\in\{1,\dots,d\}$.

First, for $\bsm=(m_1,\dots,m_d)\in\N_0^d$, we define the sets 
\[
\rho(\bsm) \,:=\, \{\bsx\in\R^d\colon \lfloor2^{m_j-1}\rfloor\le |x_j|<2^{m_j} 
\text{ for } j=1,\dots,d\}.
\]
Note that $\prod_{j=1}^d|x_j|<2^{\|\bsm\|_1}$ for all $x\in\rho(\bsm)$. 
Since $B_\scalingFactor=\scalingFactor^{1/d}B=(D_p \scalingFactor)^{1/d}\, \bm{V}^{-\top}$ we have
\[
\Nm(B_\scalingFactor) \,=\, \inf_{\bsk\in\Z^d\setminus\{0\}}\prod_{j=1}^d (B_\scalingFactor \bsk)_j
\,=\, \frac{\scalingFactor}{D_p}.
\]
This shows that $|(B_\scalingFactor(\Z^d)\setminus0)\cap\rho(\bsm)|=0$ for all 
$\bsm\in\N_0^d$ with $\|\bsm\|_1< R_\scalingFactor$, where 
\[
R_\scalingFactor \,:=\, \left\lceil \log_2\bigl(\scalingFactor/D_p\bigr) \right\rceil.
\]
Moreover, for $B_\scalingFactor\bsk\in\rho(\bsm)$, we have 
\[
\nu_{*,\bsr}(B_\scalingFactor\bsk) \,\ge\, \nu_{\bsr}(B_\scalingFactor\bsk) 
\,\ge\, \prod_{j=1}^d\max\{1,2\pi\lfloor2^{m_j-1}\rfloor\}^{r_j}
\,\ge\, 2^{r_1 m_1+\dots+r_d m_d}.
\]
Since $\rho(\bsm)$ is a union of $2^d$ axis-parallel boxes 
each with volume less than $2^{\|\bsm\|_1}$, \eqref{f10} implies that 
$\abs{B_\scalingFactor(\Z^d)\cap\rho(\bsm)} \le 2^d(D_p 2^{\|\bsm\|_1}/\scalingFactor+1)
\le 2^{d+2} 2^{\|\bsm\|_1-R_\scalingFactor}$ 
if $\|\bsm\|_1\ge R_\scalingFactor$.
Additionally, note that 
$\abs{\{\bsm\in\N_0^\eta\colon \|\bsm\|_1=\ell\}}=\binom{\eta-1+\ell}{\eta-1}$.
With $r:=r_1$ and $r':=r_{\eta+1}$, we obtain 
\[\begin{split}
\sum_{\bsk\in\Z^d\setminus0} &|\nu_{*,\bsr}(B_\scalingFactor\bsk)|^{-2}
\,\le\, \sum_{m: \|m\|_1\ge R_\scalingFactor}\,
			\abs{B_\scalingFactor(\Z^d)\cap\rho(m)}\, 2^{-2r_1 m_1-\ldots-2r_d m_d}\\
\,&\le\, 2^{d+2}\, \sum_{m: \|m\|_1\ge R_\scalingFactor}\, 
				2^{\|m\|_1-R_\scalingFactor}\, 2^{-2r (m_1+\ldots+m_\eta)-2r'(m_{\eta+1}+\ldots+m_d)} \\
\,&=\, 2^{d+2}\, \sum_{\ell=R_\scalingFactor}^\infty\sum_{m: \|m\|_1=\ell}\, 
				2^{\|m\|_1-R_\scalingFactor}\, 2^{-2r (m_1+\ldots+m_\eta)-2r'(m_{\eta+1}+\ldots+m_d)} \\
\,&=\, 2^{d+2}\, \sum_{\ell=R_\scalingFactor}^\infty\, \sum_{\ell'=0}^\ell \;
				\sum_{\substack{m_1,\dots,m_\eta:\\ \sum_{j=1}^\eta m_j=\ell-\ell'}} \;
				\sum_{\substack{m_{\eta+1},\dots,m_d:\\ \sum_{j=\eta+1}^d m_j=\ell'}}
				2^{\ell-R_\scalingFactor}\, 2^{-2r (\ell-\ell')-2r'\ell'} \\
\,&=\, 2^{d+2}\, \sum_{\ell=R_\scalingFactor}^\infty\, \sum_{\ell'=0}^\ell 
				\binom{\eta-1+\ell-\ell'}{\eta-1}\, \binom{d-\eta-1+\ell'}{d-\eta-1}\,
				2^{\ell-R_\scalingFactor-2r \ell }\, 2^{-2(r'-r)\ell'} \\
\,&\le\, 2^{d+2}\, \sum_{\ell=R_\scalingFactor}^\infty\, \binom{\eta-1+\ell}{\eta-1}\, 
				2^{\ell-R_\scalingFactor-2r \ell }\;	\sum_{\ell'=0}^\ell 
				 \binom{d-\eta-1+\ell'}{d-\eta-1}\,
				 2^{-2(r'-r)\ell'} \\
\end{split}\]
In the last estimate we used that $\binom{k+\ell}{k}\le\binom{k+\ell+1}{k}$ for 
every $k,\ell\in\N$. To bound the two sums above we use the well-known binomial identity 
\[
\sum_{\ell=0}^\infty \binom{D+\ell}{D}\, x^\ell \,=\, \frac{1}{(1-x)^{D+1}}
\]
as well as the bound 
\[
\binom{D+\ell+R}{D} \,\le\, \binom{D+\ell}{D}\,(1+R)^D
\]
for $D,\ell,R\in\N$ and $x\in\C$ with $|x|<1$.
We obtain for the second sum that
\[
\sum_{\ell'=0}^\ell \binom{d-\eta-1+\ell'}{d-\eta-1}\, 2^{-2(r'-r)\ell'} 
\,\le\, \left(1-2^{-2(r'-r)}\right)^{-(d-\eta)}
\]
and for the first sum that
\[\begin{split}
\sum_{\ell=R_\scalingFactor}^\infty\, \binom{\eta-1+\ell}{\eta-1}\, 2^{\ell-R_\scalingFactor-2r \ell } 
\,&=\, \sum_{\ell=0}^\infty\, \binom{\eta-1+\ell+R_\scalingFactor}{\eta-1}\, 2^{\ell-2r(\ell+R_\scalingFactor) } \\
\,&\le\, 2^{-2r R_\scalingFactor}\,(1+R_\scalingFactor)^{\eta-1}\,\sum_{\ell=0}^\infty\, \binom{\eta-1+\ell}{\eta-1}\, 2^{(1-2r)\ell} \\
\,&=\, 2^{-2r R_\scalingFactor}\,(1+R_\scalingFactor)^{\eta-1}\, \left(1-2^{(1-2r)}\right)^{-\eta}
\end{split}\]
for $r>1/2$.
If we use 
$\log_2\bigl(n/D_p\bigr) \le R_\scalingFactor \le 1+\log_2\bigl(n/D_p\bigr)$
we finally obtain Theorem~\ref{thm:theor-error}.\end{proof}

\section{Numerical results: Exact worst-case errors in $\mathring{H}^{\mathbf{r}}_\mix$} \label{sec:numerics}
In Section \ref{sec:TheorBounds} it has been shown that the Frolov method achieves the optimal rate of convergence in Sobolev spaces with both, uniform and 
anisotropic mixed smoothness. However, as we have seen in Section \ref{sec:polynomials}, there are different ways to choose the polynomials, 
which significantly influence the numerical performance. Therefore, even though the asymptotic 
convergence rate of all (admissible) Frolov cubature rules have the optimal order
$\mathcal{O} (N^{-r} (\log N)^{(d-1)/2})$ for uniform smoothness $f$,
there might be huge constants involved. In order to investigate the influence of different Frolov polynomials on the preasymptotic 
behavior of the integration error, we use a well-known technique for reproducing kernel Hilbert spaces 
to compute the worst-case error explicitly. This supplements the theoretical bounds from Section \ref{sec:TheorBounds}.
Moreover, we compare the worst-case errors of Frolov cubature, the sparse grid method and quasi--Monte 
Carlo methods in $\mathring{H}^r_\mix$.

\subsection{Exact worst-case errors via reproducing kernels}
The worst-case error of any linear cubature rule $Q_N(f) = \sum_{i=1}^N w_i f(\bsx_i)$ with prescribed 
weights and nodes can be computed \emph{exactly} via the norm of the error functional $R_N(f) := I_d(f) - Q_N(f)$, cf. Eq. \eqref{eqn_functional_norm}. Applying $R_N$ to both components of the kernel $ \mathring{K}_{d}^\bsr(\bsx, \bsy)$, the well-known formula for the (absolute) worst-case error is obtained, i.e
\begin{equation} \label{eqn_wce_formula}
	\begin{aligned}
		\sup_{\|f\|_{\mathring{H}^\bsr_\mix} \leq 1} |R_N(f)|^2 & = \int\displaylimits_{[0,1]^d} \int\displaylimits_{[0,1]^d} \mathring{K}_{d}^\bsr(\bsx, \bsy) \, \rd \bsx \rd \bsy - 2 \sum_{i=1}^N w_i \int\displaylimits_{[0,1]^d} \mathring{K}_{d}^\bsr(\bsx_i, \bsy) \, \rd \bsy \\
		& \quad \quad + \sum_{i=1}^N \sum_{j=1}^N w_i w_j \mathring{K}_{d}^\bsr(\bsx_i, \bsx_j) .
	\end{aligned}
\end{equation}

Often, \eqref{eqn_wce_formula} is normalized with respect to norm of 
$I_d$ in the dual-space $(\mathring{H}^\bsr_\mix)^\star$, i.e. 
\eqref{eqn_wce_formula} is divided by 
$\|I_d\|_{(\mathring{H}^\bsr_\mix)^\star} = (\int_{[0,1]^d}\int_{[0,1]^d}  \mathring{K}_{d}^\bsr(\bsx, \bsy) \, \rd \bsx \rd \bsy )^{1/2}$, 
cf. \eqref{eqn_initial_error}.
The resulting quantity is called \emph{normalized worst-case error}.


In order to evaluate \eqref{eqn_wce_formula} for an arbitrary given cubature rule we use the closed-form
representation of the kernel $\mathring{K}^{\bsr}_{d}$ from Theorem \ref{repr_kernel} as well as the closed-form representation of the Riesz-representer \eqref{eqn_riesz_representer}.

Besides Frolov cubature rules, we will consider the sparse grid construction, which goes back to 
Smolyak \cite{Smolyak:1963}, and also higher-order quasi--Monte Carlo integration \cite{DiPi10}. 
Examples for the different point constructions are given in Figure \ref{fig_cubature}. Their properties will be discussed below.

The Frolov points are generated using our newly developed  Algorithm \ref{QRassemble}. The resulting points obtained by the improved polynomial construction can also be downloaded from \url{http://wissrech.ins.uni-bonn.de/research/software/frolov}.

\begin{figure}[t]
 \includegraphics[width=0.32\linewidth]{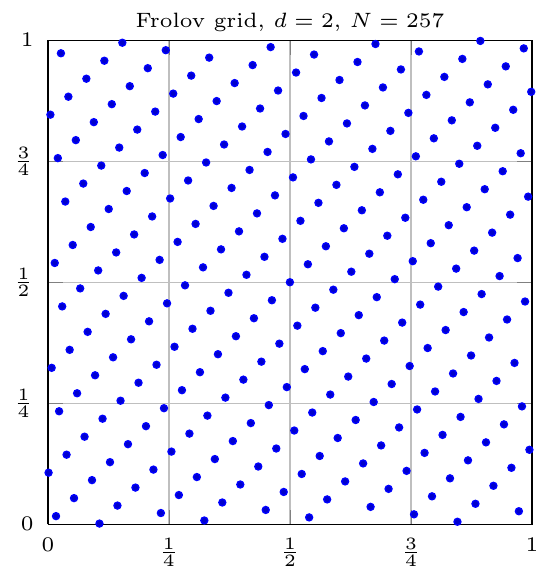}
 \includegraphics[width=0.32\linewidth]{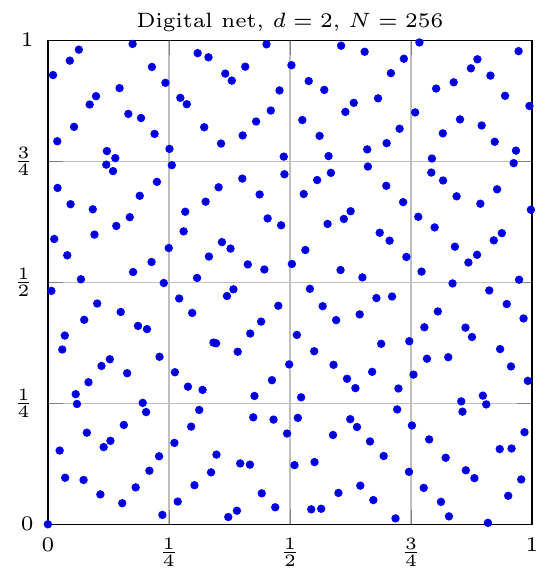}
 \includegraphics[width=0.32\linewidth]{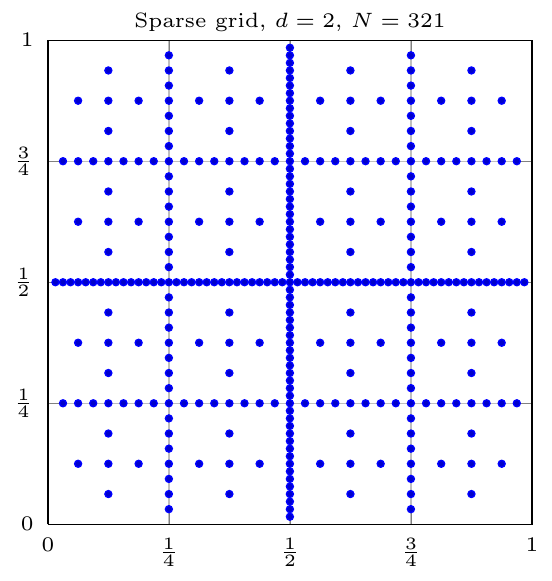}
 
 \caption{A Frolov lattice (left), an order-$2$ digital net (middle) and a zero  boundary sparse grid (right).} \label{fig_cubature}
\end{figure}

\subsection{Uniform mixed smoothness}
As a first step we compare worst-case errors for cubature formulas that
are known to work well in periodic Sobolev spaces,
of which  $\mathring{H}^r_\mix$ is a subset. These are different Frolov
cubature rules, that are based on different choices of the
generating polynomial. In the following, ''Classical Frolov'' will refer to the classical generating polynomial in \eqref{eqn_classical_polynomial}, while ''Improved Frolov'' will refer to the lattices that are generated by the improved polynomials from Section \ref{sec:polynomials}. Moreover, we
consider the sparse grid method that is based on the
trapezoidal rule, see Appendix \ref{sec_sparsegrid}. Due to the zero-boundary condition in $\mathring{H}^r$, all points with one component equal to zero are left out, cf. Figure \ref{fig_cubature}.\footnote{This is similar to the open trapezoidal rule which, however, uses different weights, cf. \cite{Gerstner.Griebel:1998}.}
It achieves a
convergence rate of
order $\mathcal{O}(N^{-r} (\log N)^{(d-1)(r+1/2)})$ in
$\mathring{H}^r_\mix$, which is best possible for a sparse grid method, cf. Theorem \ref{cor_sg} below. 
As an example for a higher order quasi--Monte Carlo method we use a
digital net of order $2$ that is obtained by interlacing
the digits of a $(2d)$-dimensional Niederreiter-Xing net. This is
obtained by using the implementation of Pirsic  \cite{Pirsic} of 
Xing-Niederreiter sequences \cite{NiederreiterXing} for rational places
in dimension $2d-1$. These are known to yield smaller $t$-values than
e.g. Sobol- or classical Niederreiter-sequences \cite{Dick2008572}.
Then, a $2d$-dimensional digital net is obtained by employing the
sequence-to-net propagation rule, cf. \cite{DiPi10,Niederreiter1996241} for
more details.
It is known that order-$2$ nets yield the optimal rate of convergence in periodic Sobolev spaces with bounded mixed derivatives of order $r < 2$, see \cite{Hinrichs.Markhasin.Oettershagen.ea:2016} and also \cite{GodaQMC}, since $\mathring{H}^r_\mix \subset H^r_\mix(\mathbb{T}^d)$.

Moreover, in the bivariate setting we also consider the Fibonacci
lattice, which is not just known to be an order-optimal cubature rule
for periodic Sobolev spaces with dominating mixed smoothness \cite{DTU16}, but also
represents the best possible point set for quasi -- Monte Carlo
intergation in this space, at least for small point numbers
\cite{Hinrichs.Oettershagen:2016}.

\begin{figure}
\centering
 \includegraphics[width=0.45 \linewidth]{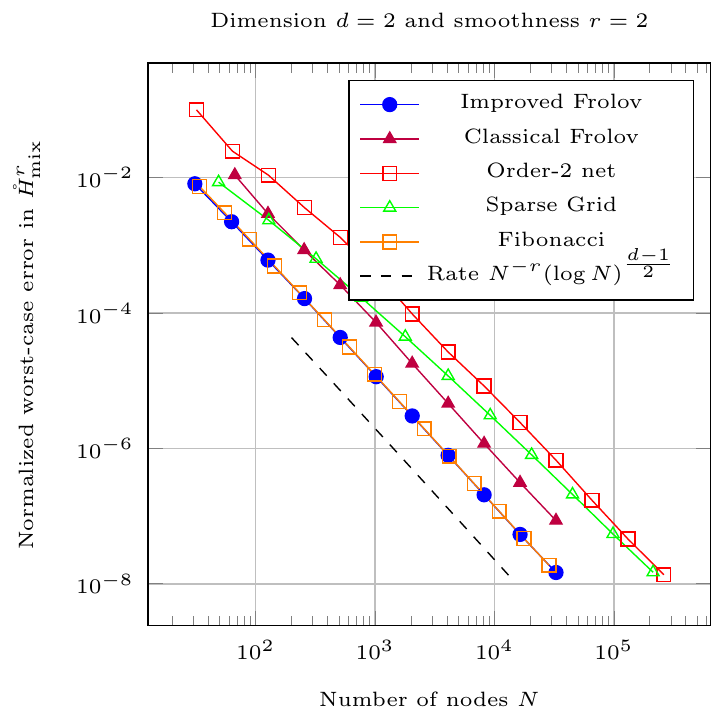}\,
 \includegraphics[width=0.45 \linewidth]{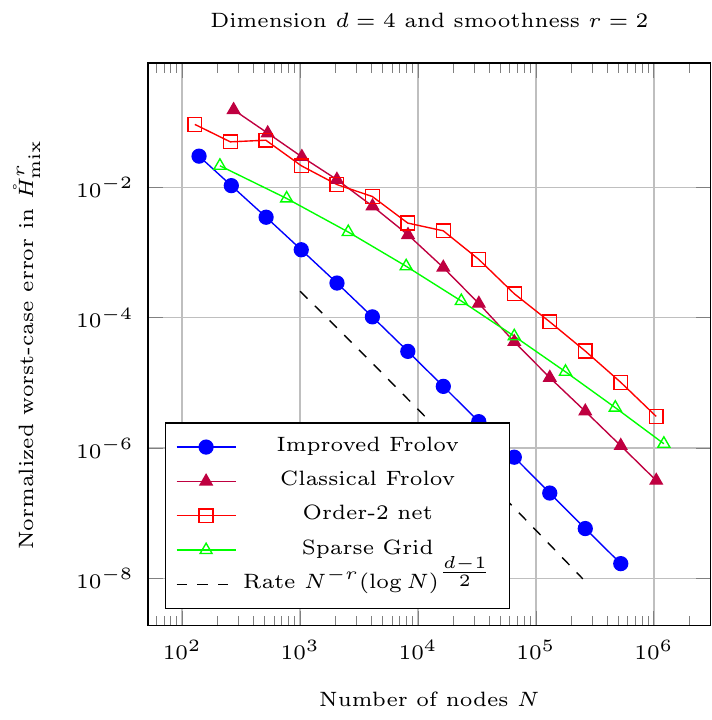}
 \caption{Worst-case errors for different cubature rules for uniform mixed smoothness $r=2$ in
dimensional $d=2$ (left) and dimension $d=4$ (right).} \label{fig_d2_d4_wce}
\end{figure}

In the left-hand-side picture of Figure \ref{fig_d2_d4_wce}, the
worst-case errors for smoothness $r=2$ are computed in dimension $d=2$.
Clearly, the Frolov lattice based on the improved polynomial performs
best in $\mathring{H}^r_\mix$. Of similar quality is the Fibonacci
lattice and the classical Frolov lattice is slightly worse. The sparse
grid also achieves the optimal main rate of $N^{-r}$, but it is known
that the exponent of its logarithm is smoothness dependent. This is also
apparent in Figure \ref{fig_d2_d4_wce}, where the sparse grid has an
asymptotic behavior that is inferior to all the other considered
methods. On the right-hand-side of Figure \ref{fig_d2_d4_wce}, the
worst-case errors for smoothness $r=2$ are computed in dimension $d=4$.
Here, the Fibonacci lattice is not considered. However, for all the
other methods we note that the picture does not change much, compared to
the case $d=2$. As before, the improved Frolov method performs best and
the classical Frolov  obtains the same optimal asymptotic convergence
rate but a substantially worse constant. This effect is now much
stronger than in the bivariate setting, i.e. the classical Frolov lattice has
a worst-case error that is about two magnitudes larger than the one of
the improved Frolov lattice. Moreover, the order-$2$ digital net seems
to be competitive too, albeit with a substantially larger constant and
longer pre-asymptotic regime. Again, the worse logarithmic exponent of
the sparse grid method can be clearly observed.

In the Figures \ref{fig_r1}, \ref{fig_r2} and \ref{fig_r3} the influence of
the dimensionality and the smoothness onto the performance of the Frolov
cubature method is considered in more detail.
As an example for a cubature method with a less than optimal complexity, the sparse grid method is also included.
Especially the classical construction suffers from a strong growth of the constant as the
dimensionality increases. Also, the pre-asymptotic regime seems to last
longer. This effect can so far not be thoroughly explained by the
existing theory. In dimension $d=7$, the classical Frolov construction
needs more than $10^6$ points to achieve the error level of the
zero-algorithm, i.e. normalized worst-case error $1$. Note at this
point, that all given errors are normalized worst-case errors, which
can, for non optimally weighted cubature rules, be substantially larger
than $1$.
It is apparent that the classical Frolov method is practically useless
in dimension $d \geq 5$, due to its unfavorable pre-asymptotic behavior.
Our new approach, however, shows a much better dependence onto the
dimensionality and certainly allows the treatment of
moderate-dimensional integrals from Sobolev spaces with dominating mixed
smoothness of uniform type.

Moreover, we observe the universality of Frolov's method, i.e. without adaption to the respective parameters
it achieves the best possible rate of convergence in every $\mathring{H}^r_\mix$, 
$r \in \{1,2,3\}$.

\subsection{Anisotropic mixed smoothness}

%

\begin{figure}[t]
\centering
 \includegraphics[width=0.41\linewidth]{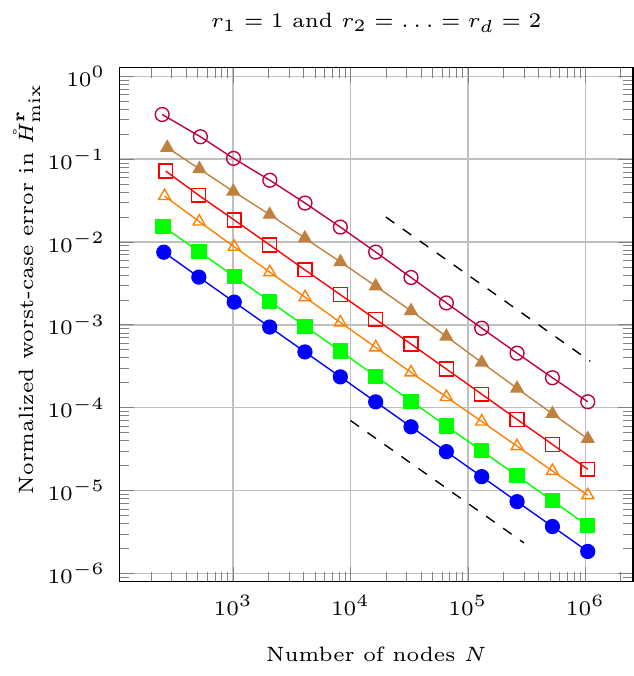} \, \includegraphics[width=0.56\linewidth]{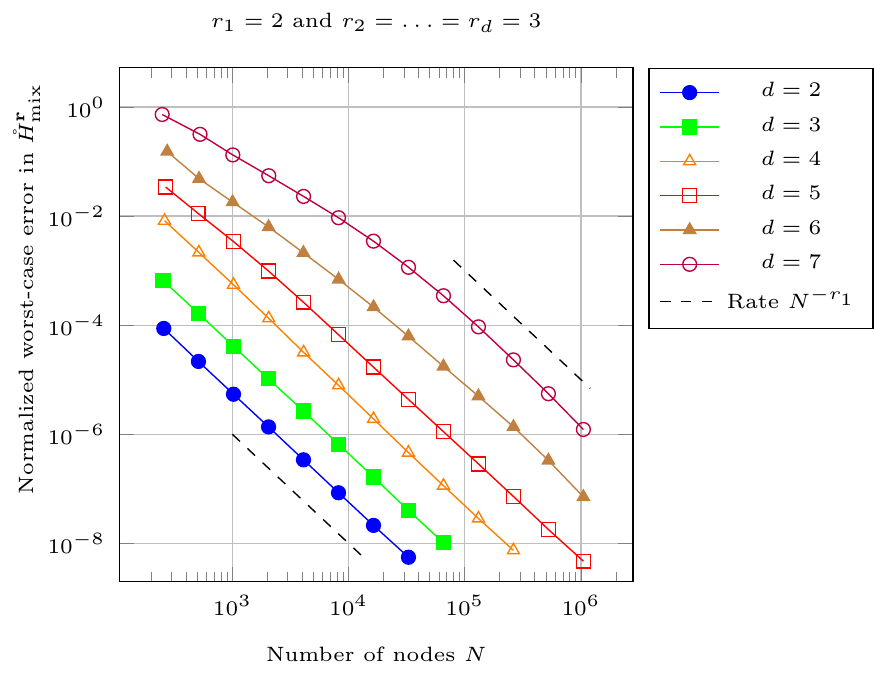}
 
 \caption{Worst-case errors for uniform mixed smoothness in various dimensions. Left-hand side: $r_1=1$ and $r_2 = \cdots = r_d = 2$. Left-hand side: $r_1=2$ and $r_2 = \cdots = r_d = 3$.} \label{fig_aniso}
\end{figure}

It has been shown in Theorem \ref{thm:theor-error} 
that in Sobolev spaces with dominating mixed smoothness of different orders in each direction, only the lowest smoothness and associated dimension enters the error estimate. In order to make this phenomenon visible from a numerical perspective, we  compute explicit worst-case errors in 
\[
 \mathring{H}^\bsr_\mix = \mathring{H}^{r_1} \otimes \cdots \otimes \mathring{H}^{r_d} ,
\]
where $r_1 = r$ and $r_2 = r_3 = \cdots = r_d = r+1$. Then, Theorem \ref{thm:theor-error} predicts that the worst-case error asymptotically behaves like in the univariate setting, i.e. decays at a rate of $\mathcal{O}(N^{-r})$. The question that is investigated in Figure \ref{fig_aniso} is how long it takes to overcome the preasymptotic regime until this favorable convergence rate becomes visible.

On the left-hand-side of Figure \ref{fig_aniso}, i.e. for $r=1$, already with less than $3000$ points the Frolov method follows the asymptotic regime of $N^{-1}$ in all the considered cases $d \in \{2,3,\ldots, 7\}$.

In contrast, on the right-hand-side of Figure \ref{fig_aniso}, i.e. for $r=2$, the dimension seems  to have a much larger impact onto the length of the sub-optimal preasymptotic regime. For example, in $d=7$ the $N^{-2}$-rate becomes visible only when the number of points $N$ is larger than $\approx 10^5$.

We remark that the sparse grid method is also able to deal with anisotropic mixed smoothness vectors 
$\bsr = (r_1,\ldots, r_d)$. Then, however, the construction needs to be adjusted to 
the smoothness vector which has to be known in advance, 
see \cite[pp. 32,36,72]{Tem86}, the recent survey \cite[Sect.\ 10.1]{DTU16} and the 
references therein. The resulting sparse grid construction therefore is not a universal cubature 
formula.\footnote{Note that it is also possible to construct dimension-adaptive spare grids, which are are able to detect the smoothness vector in the process of approximation adaptively, cf. \cite{Gerstner.Griebel:2003}.}

However, both plots in in Figure \ref{fig_aniso} were computed with the exact same set of Frolov points, which automatically benefit from the anisotropic smoothness that is present in a given integration problem, i.e. in this case $r=1$ or $r=2$. Therefore, it is not necessary to estimate the smoothness of the integrand and tune the method appropriately.

%

%

\section*{Acknowledgement}
T.U. wishes to thank Winfried Bruns (Osnabrueck) for several fruitful discussions.
T.U. and J.O. gratefully acknowledge support by the German Research Foundation (DFG) Ul-403/2-1, GR-1144/21-1 and the Emmy-Noether programme, Ul-403/1-1.

\newpage
\appendix
\section{Appendix: Sparse grid cubature in $H^r_\mix(\mathbb{T}^d)$}\label{sec_sparsegrid}
Let
\begin{equation}\label{trapez}
 Q_N(f) := \sum_{j=0}^{N-1}\frac{1}{N} f\Big( \frac{j}{N} \Big)
\end{equation}
denote the uniformly weighted $N$-point trapezoidal rule. It is known that it achieves the optimal rate of convergence $N^{-r}$
in the periodic Sobolev space $H^r(\mathbb{T}), r \in \N$ (our proof below also works for the univariate case). 
In order to obtain a multivariate integration method we define the hierarchical quadrature rules
\begin{equation} \label{eqn_delta}
 \Delta_k := \Delta_k(Q) := Q_{2^{k}} - Q_{2^{k-1}} \quad \text{ for all } k = 1, 2, \ldots 
\end{equation}
and $\Delta_0 = Q_1$. Their tensor product is denoted by $\Delta_{\bsk} := \bigotimes_{j=1}^d \Delta_{k_j}$, $\bsk \in \mathbb{N}_0^d$.
The sparse grid cubature rule of level $L\in \N$ is then given by
\begin{equation} \label{eqn_sg}
 Q^{sg}_L := \sum_{|\bsk|_1 \leq L} \Delta_{\bsk} ,
\end{equation}
with multi-indices $\bsk = (k_1,\ldots,k_d) \in \N_0^d$. The cubature rule 
$Q^{sg}_L$ uses 
\begin{equation}\label{number_nodes}
   N_L = \sum_{|\bsk|_1 \leq L} \prod_{j=1}^d 2^{k_j-1} = \mathcal{O}(2^L \cdot L^{d-1})
\end{equation}
function values combined with non-equal weights. The following theorem gives the well-known error bound in $H^r_{\mix}(\mathbb{T}^d)$. For the convenience of the 
reader we will also give a proof. 


\begin{theorem} \label{cor_sg}
Consider the sparse grid cubature rule $Q^{sg}_L$ as it is defined in 
\eqref{eqn_delta} and \eqref{eqn_sg} based on the univariate trapezoidal rule 
$\eqref{trapez}$. The worst-case integration error of $Q^{sg}_L$ in $H^r_{\mix}(\mathbb{T}^d)$ can be bounded by
\begin{equation}\label{sg_bound}
 \sup_{\|f\|_{H^r_\mix} \leq 1} \left| \int_{[0,1]^d}f(\bsx)\rd\bsx - Q^{sg}_L(f) \right| \asymp N^{-r} (\log N)^{(d-1)(r+1/2)} ,
\end{equation}
where $N = N_L$ denotes the number of points used by $Q^{sg}_L$.
\end{theorem}
\begin{proof} The lower bound follows from \cite[Thm.\ 5.2]{DuUl15}. Note, that the lower bound also holds true 
for the smaller space $\mathring{H}^r_{\mix}(\mathbb{T}^d)$ since the constructed fooling functions also belong to this space. 
For the upper bound we use the detour to sampling recovery. In the recent paper 
\cite[Thm.\ 4.7, 4.8, 5.13, 5.14]{ByUl17} it has been observed that nested trigonometric interpolation operators 
\begin{equation}\label{f100}
I_{2^k}[f](x) = \frac{1}{2^k}\sum\limits_{u=0}^{2^{k}-1} f\Big(\frac{u}{2^k}\Big) \mathcal{D}_{2^k}^1\Big(x-\frac{u}{2^k}\Big)
\quad,\quad k=0,1,2,...,
\end{equation}
based upon the modified (nested) Dirichlet kernel 
$\mathcal{D}_{2^k}^1(x) := \mathcal{D}_{2^{k-1}}(x) - e^{2\pi i2^{k-1}x}$ may be used to characterize $H^r_{\mix}(\mathbb{T}^d)$.
In fact, the tensor products $\Delta_{\bsk}(I)$, $\bsk \in \N_0^d$, are defined analogously to 
\eqref{eqn_delta} using this time \eqref{f100} (note that $\mathcal{D}_{1}^1(x) \equiv 1$).
Then we have 
\begin{equation}\label{charact}
   \|f\|_{H^r_{\mix}}^2 \asymp \sum\limits_{\bsk \in \N_0^d} 2^{2r|\bsk|_1}\|\Delta_{\bsk}(I)[f]\|^2_2\,.
\end{equation}
See also \cite[Prop.\ 3.3]{ByDuSiUl16} for the classical (non-nested) trigonometric interpolation. 
The associated sparse grid interpolation operator $I^{sg}_L$ is defined in the same way as above in \eqref{eqn_sg}. 
Now we argue similar as in \cite[Thm.\ 5.4]{ByDuSiUl16}. Indeed, H\"older's 
inequality together with \eqref{charact} gives 
\begin{equation}\label{A7}
  \begin{split}
    \|f-I^{sg}_L[f]\|_2 &\leq \Big(\sum\limits_{|\bsk|_1>L} 2^{-2|\bsk|_1 r}\Big)^{1/2}\cdot 
    \Big(\sum\limits_{|\bsk|_1>L} 2^{2r|\bsk|_1}\|\Delta_{\bsk}(I)[f]\|^2_2\Big)^{1/2}\\
    &\leq 2^{-rL}L^{(d-1)/2}\|f\|_{H^r_{\mix}}\,.
  \end{split}
\end{equation}
Noting further that 
$$
    Q_L^{sg}(f) = \int_{[0,1]^d} I^{sg}_L[f](\bsx)\rd\bsx
$$
we have by H\"older's inequality and \eqref{A7}
$$
    \Big|\int_{[0,1]^d}f(\bsx)\rd\bsx - Q^{sg}_L(f)\Big| \leq \|f-I^{sg}_L[f]\|_2 \leq 2^{-rL}L^{(d-1)/2}\|f\|_{H^r_{\mix}}\,,
$$    
see also \cite[Rem.\ 8.9]{DTU16}\,. Finally, the bound \eqref{sg_bound} follows from \eqref{number_nodes}\,.
\end{proof}

\begin{remark} The above multivariate cubature rule on the sparse grid uses a number of nodes on the boundary of $[0,1]^d$ 
which are not needed when dealing with functions from $\mathring{H}^r_{\mix} \subset H^r_{\mix}(\mathbb{T}^d)$\,. However, as already mentioned in the proof 
of Theorem \ref{cor_sg}, with respect to the asymptotic rate of convergence 
we can not do essentially better. However, to do a fair cost comparison for the different methods considered in Section 7 
we only counted the interior nodes (see the diagrams above, e.g. Figure \ref{fig_d2_d4_wce}). 
\end{remark}

\begin{figure}
\centering

 \includegraphics{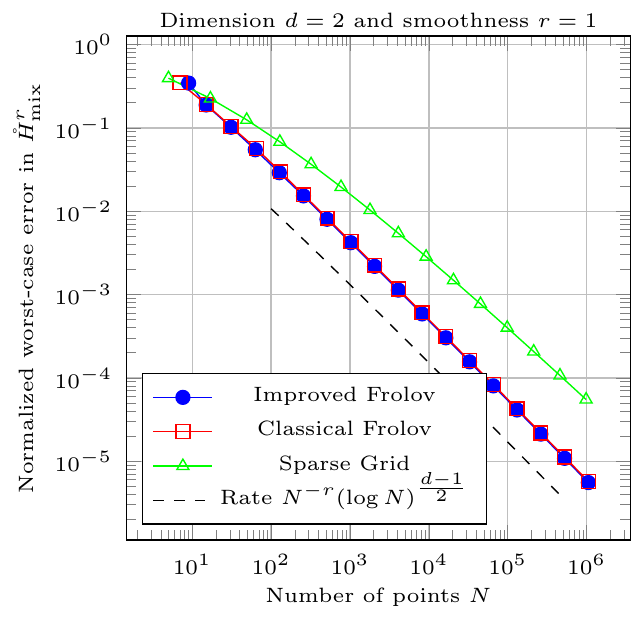}
 \,
 \includegraphics{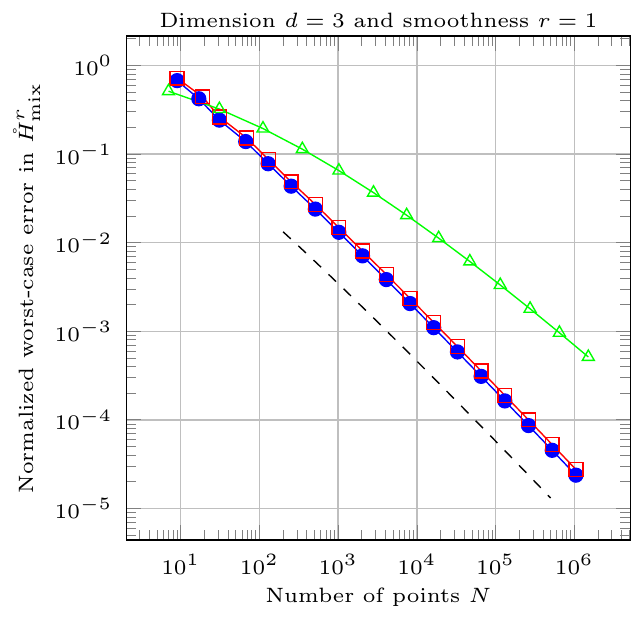}\\
 \includegraphics{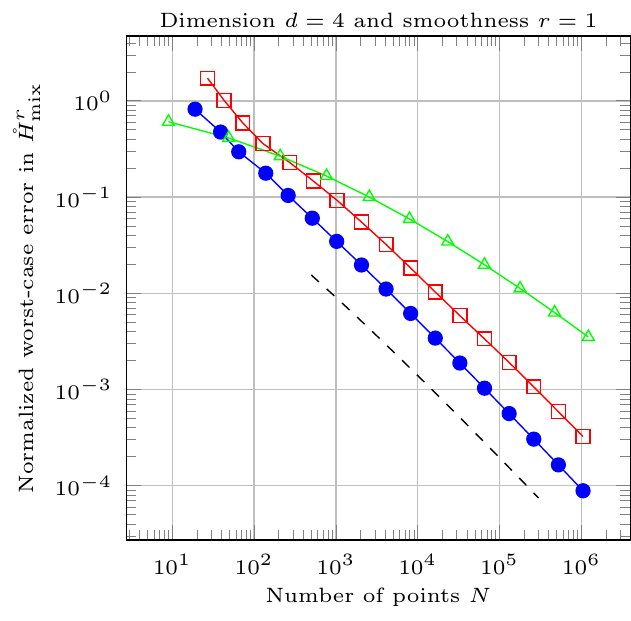}
 \,
 \includegraphics{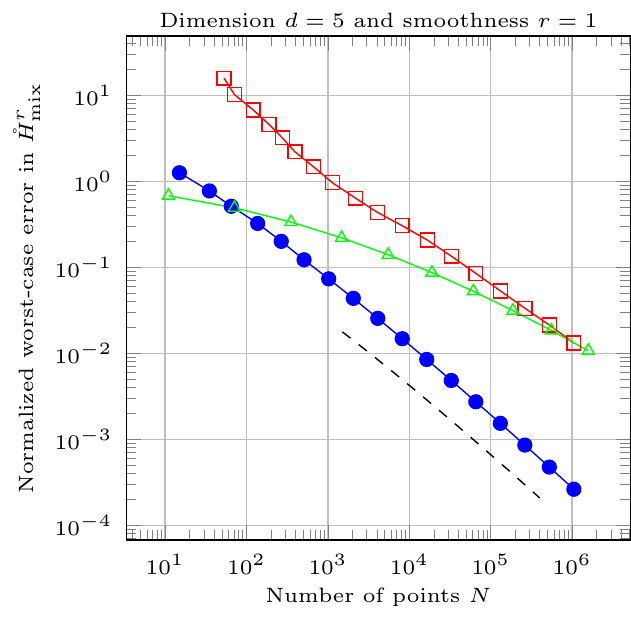}\\
 \includegraphics{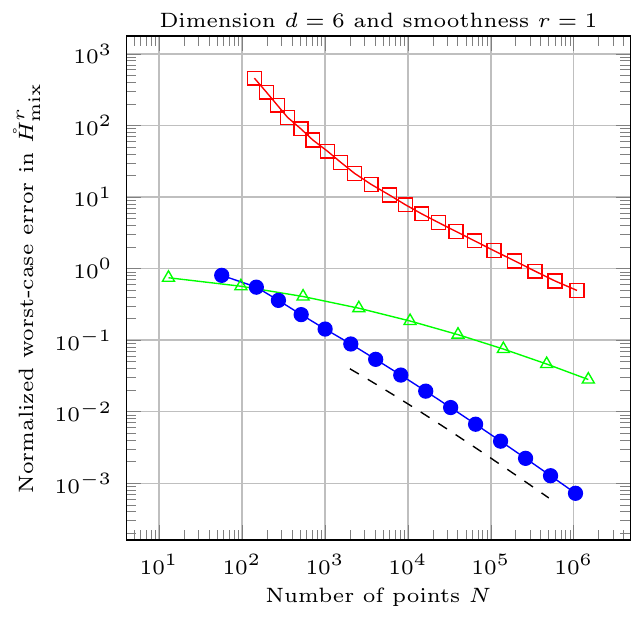}
 \,
 \includegraphics{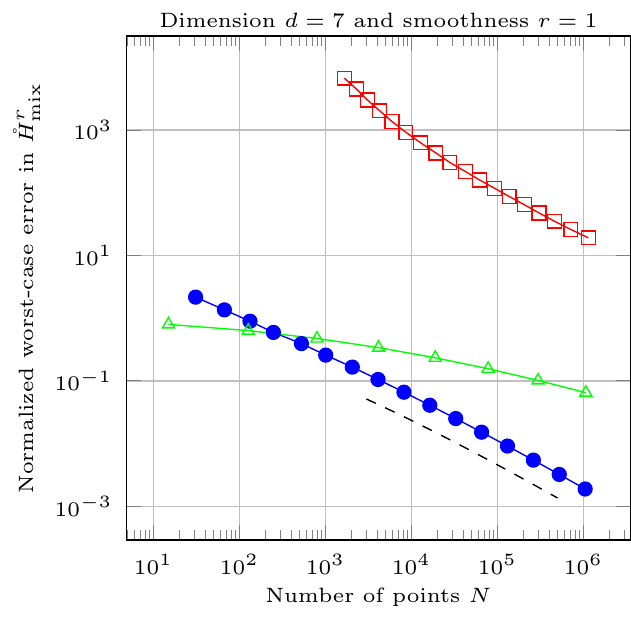} 
 \caption{Normalized worst-case errors for uniform smoothness parameter $r=1$ in dimension $d \in \{2,3,4,5,6,7\}$ for different Frolov constructions and sparse grids.} \label{fig_r1}
\end{figure}

\begin{figure}
\centering

\includegraphics[width=0.45\linewidth]{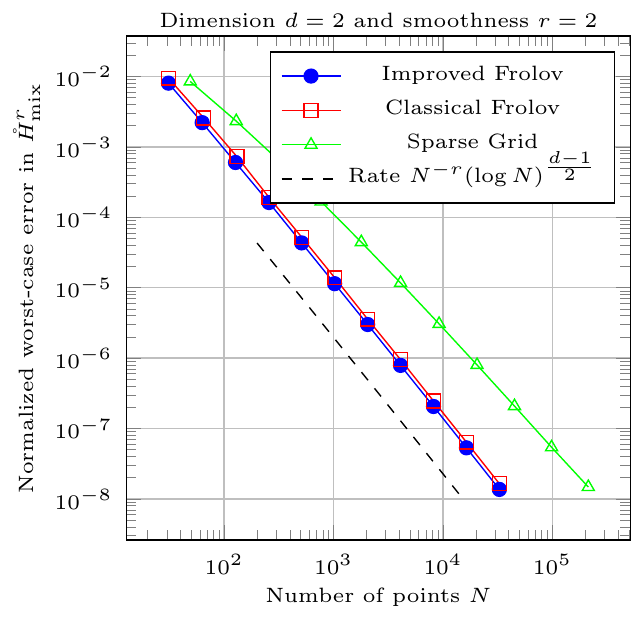} \, \includegraphics[width=0.45\linewidth]{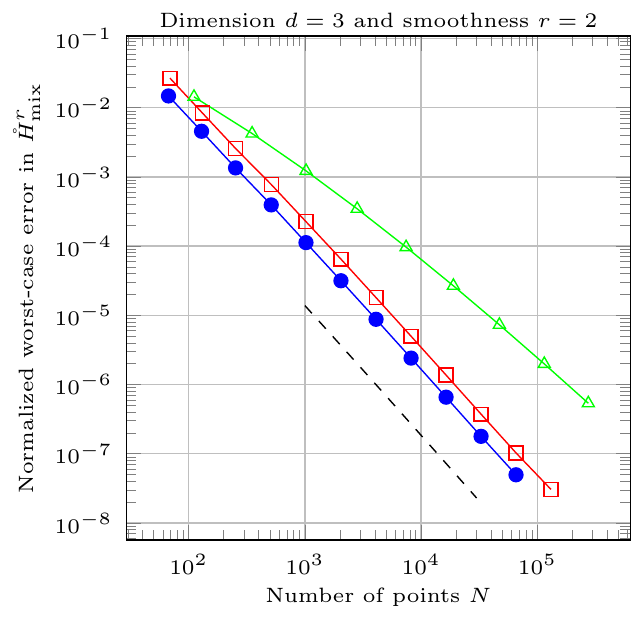}\\
 \includegraphics[width=0.45\linewidth]{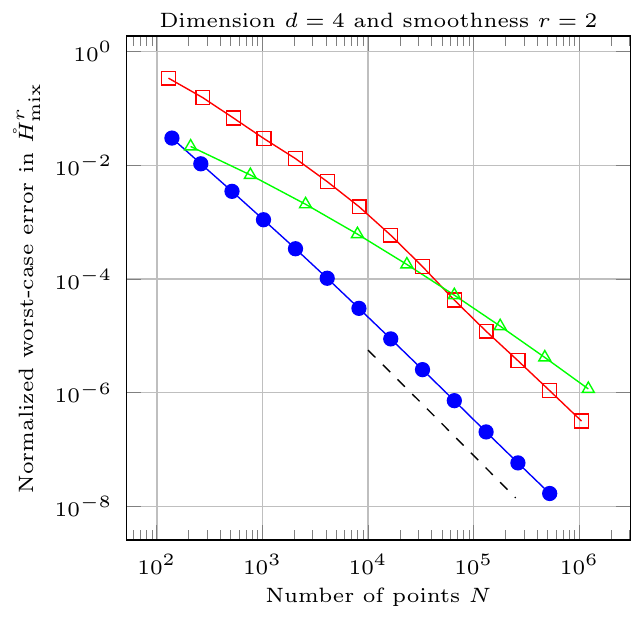} \, \includegraphics[width=0.45\linewidth]{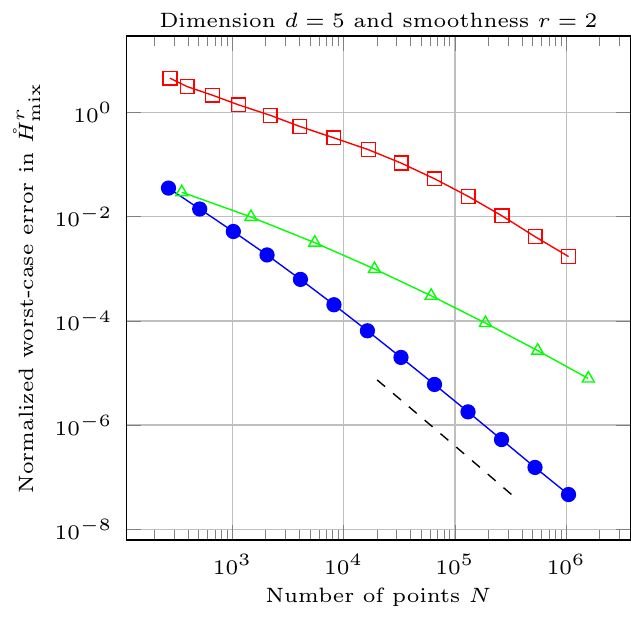}\\
 \includegraphics[width=0.45\linewidth]{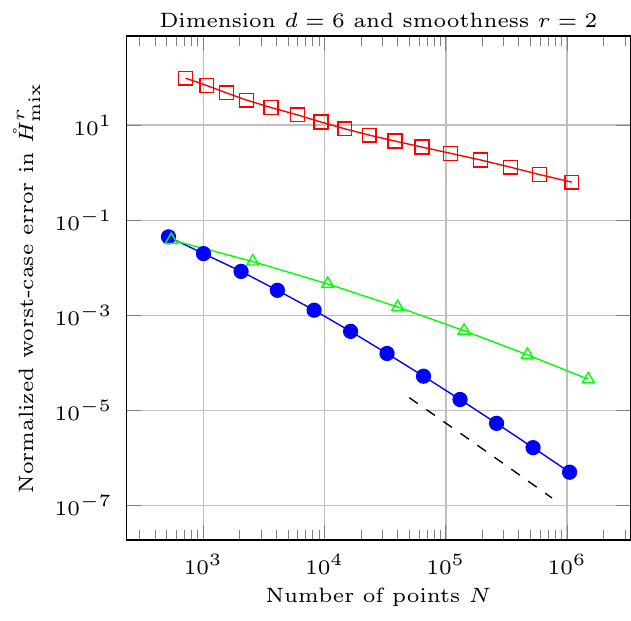} \, \includegraphics[width=0.45\linewidth]{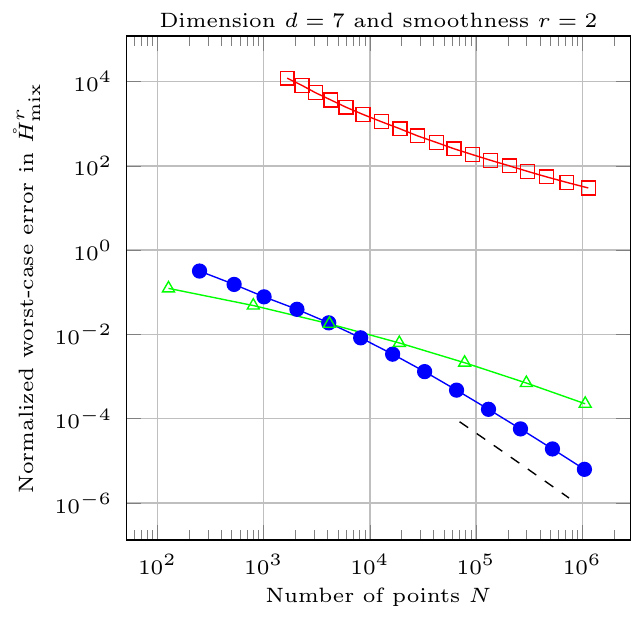}\\
 
 \caption{Normalized worst-case errors for uniform smoothness parameter $r=2$ in dimension $d \in \{2,3,4,5,6,7\}$ for different Frolov constructions and sparse grids.} \label{fig_r2}
\end{figure}

\begin{figure}
\centering

 \includegraphics {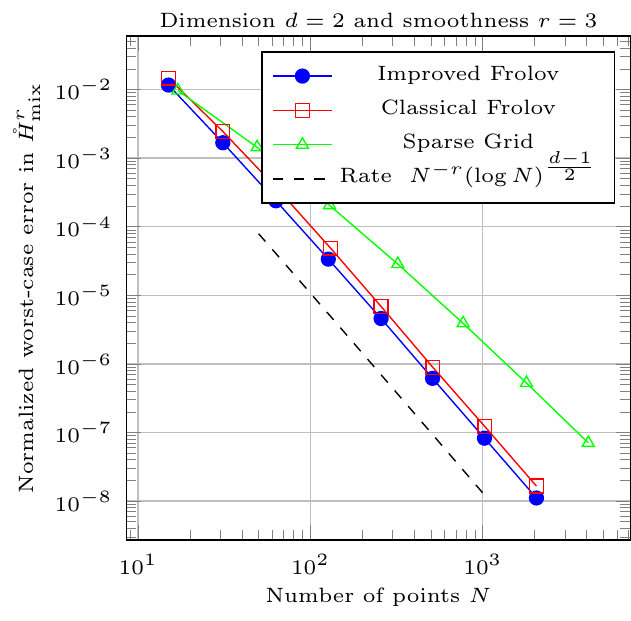} \, \includegraphics {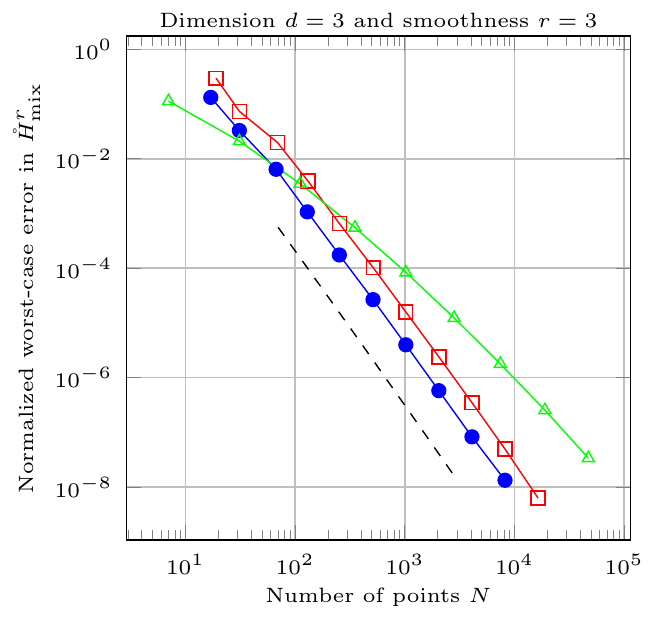}\\
 \includegraphics {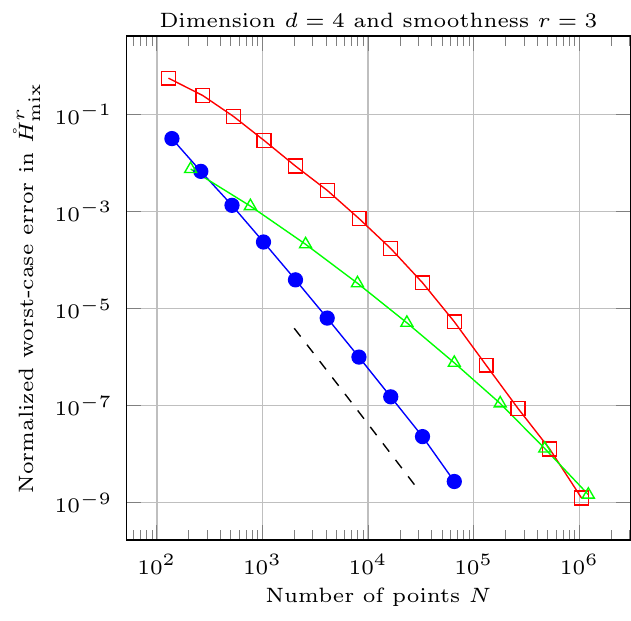} \, \includegraphics {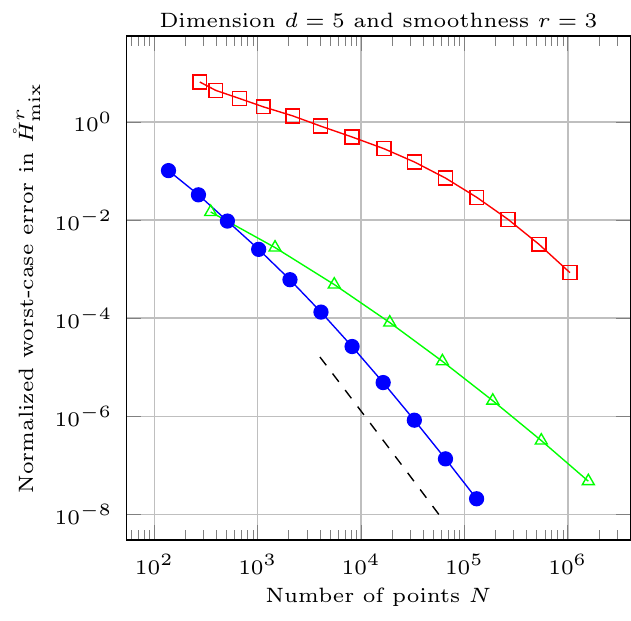}\\
 \includegraphics {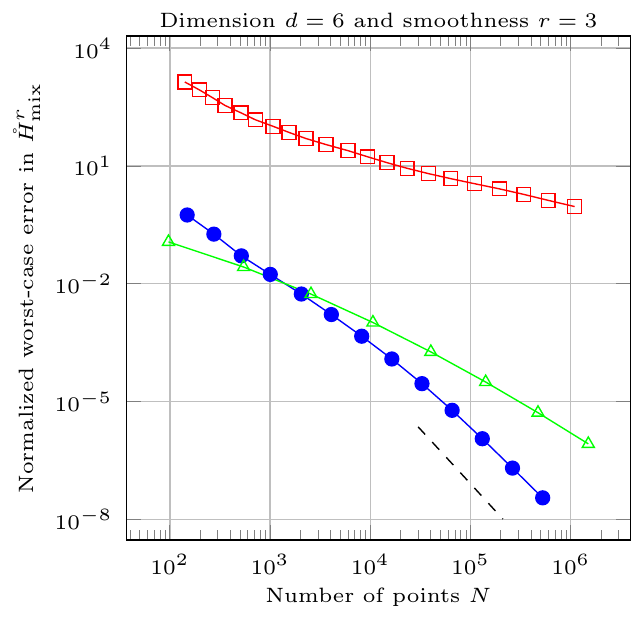} \, \includegraphics {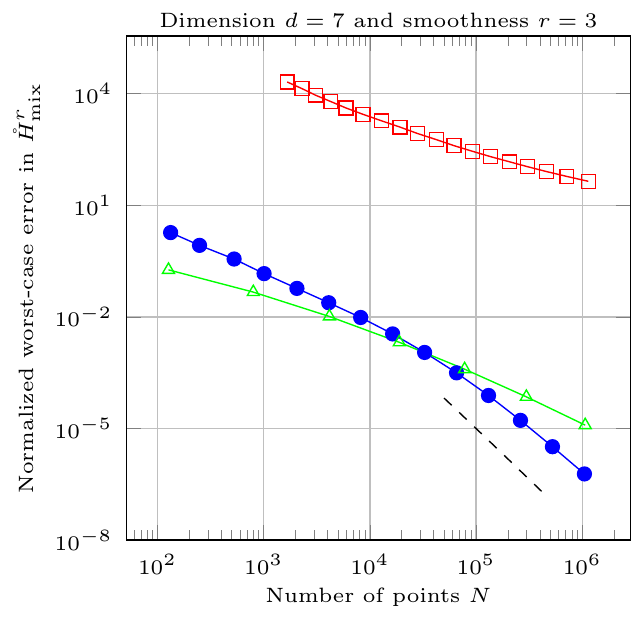}\\
 
 \caption{Normalized worst-case errors for uniform smoothness parameter $r=3$ in dimension $d \in \{2,3,4,5,6,7\}$ for different Frolov constructions and sparse grids.} \label{fig_r3}
\end{figure}

\pagebreak

\bibliography{literatur}

\end{document}